\documentclass[letterpaper,11pt,reqno]{amsart} 

\include{style}

\begin{document}
\title[Heavy repulsion of clusters]{Heavy repulsion of clusters in Bernoulli percolation}
\author[Bell]{Sasha Bell}
\address{
Sasha Bell\\
Mathematics and Statistics Department \\
McGill University \\
Montr\'{e}al, QC \\
Canada
}
\email{sasha.bell@mail.mcgill.ca}

\author[Chu]{Tasmin Chu}
\address{
Tasmin Chu\\
Division of Physics, Mathematics and Astronomy \\
California Institute of Technology \\
Pasadena, CA \\
USA
}
\email{tlchu@caltech.edu}

\author[Rodgers]{Owen Rodgers}
\address{
Owen Rodgers\\
Mathematics and Statistics Department \\
McGill University \\
Montr\'{e}al, QC \\
Canada
}
\email{owen.rodgers@mail.mcgill.ca}

\author[Terlov]{Grigory Terlov}
\address{
Grigory Terlov \\
Department of Statistics and Operations Research\\
University of North Carolina\\
Chapel Hill, NC \\
USA
}
\email{gterlov@unc.edu}

\author[Tserunyan]{Anush Tserunyan}
\address{
Anush Tserunyan \\
Mathematics and Statistics Department \\
McGill University \\
Montr\'{e}al, QC \\
Canada
}
\email{anush.tserunyan@mcgill.ca}

\begin{abstract} 
We consider Bernoulli$(p)$ percolation on (non)unimodular quasi-transitive graphs and prove that, almost surely, for any two clusters $C$ and $C'$, the set of vertices in $C$ within distance one of $C'$ is light, i.e.\ it has finite total modular weight. 
This is a significant step towards resolving a longstanding question of H\"aggstr\"om, Peres, and Schonmann, and a generalization of a theorem of Tim\'ar, who proved the same result in the unimodular setting.

We present two proofs. 
The first one is a hyperfiniteness argument that is specific to Bernoulli$(p)$ percolation and is new even in the unimodular case; the second is more general and follows Tim\'ar's approach, but requires the development of weighted analogues of several classical unimodular results.
The new insight in both proofs is nonhyperfiniteness, which we obtain from the recently developed theory of measure class preserving equivalence relations and graphs.
In particular, along the way we establish a generalization of a theorem of Benjamini, Lyons, and Schramm on the existence of an invariant random subgraph with positive weighted Cheeger constant, which is of independent interest.
\end{abstract}

\maketitle

\tableofcontents

%%%%%%%%%%%%%%%%%%%%%%%%%%%%%%%%%%%%%%%%%%%%%%%%%%%%%%%%%%%%%%%%%%%%
%
\section{Introduction and main results}
%
%%%%%%%%%%%%%%%%%%%%%%%%%%%%%%%%%%%%%%%%%%%%%%%%%%%%%%%%%%%%%%%%%%%%
In a seminal 1996 paper \cite{BSbeyond}, Benjamini and Schramm laid the groundwork for modern percolation theory by extending its scope beyond lattices and trees, traditionally of interest in physics, to encompass general \textit{locally finite, quasi-transitive, connected} graphs. 
Since then, most of the theory has been developed for the subclass of unimodular graphs (e.g. Cayley graphs of finitely generated groups) due to the additional symmetry encoded in unimodularity.
In contrast, nonunimodular quasi-transitive graphs exhibit an inherent asymmetry, which can be quantified by a relative weight function on the vertices, derived from the Haar modulus \cite{Haggstrom99,BLPS99inv,Timar06nonu,LyonsBook}.
As a consequence, most results in the nonunimodular setting have been obtained using a rather different set of techniques, tailored to exploit this additional structure.
Although said asymmetry often complicates the analysis, it also provides additional structure. 
In a breakthrough work \cite{Hutchcroft20}, Hutchcroft leveraged the existence of a nontrivial weight function for graphs whose automorphism group admits a nonunimodular quasi-transitive closed subgroup, to compute values of the critical exponents and establish a nontrivial nonuniqueness phase.

Furthermore, there is intrinsic interest in developing a theory of percolation on nonunimodular graphs that is coherent with the well-established theory for unimodular graphs. 
This endeavor is particularly timely in view of recent parallel advances in measured group and graph theory \cite{CTT22,AnushRobin}. 
Since unimodular techniques do not take into account the behavior of the weight function, such extensions are often done by introducing a weighted variant of a relevant notion to enable the rest of the machinery to go through; e.g.\ \cite{BLPS99inv,Pengfei18,CTT22,wamen}. 
In this paper, we build on these recent insights to resolve one of the remaining cases of a question posed by Häggström, Peres, and Schonmann \cite[p.87]{Haggstrom99} concerning  cluster repulsion. 

To state our main result we briefly introduce percolation terminology. 
In this paper, by a \textit{graph} we mean a locally finite connected undirected simple graph.
As above, a \textbf{bond percolation} process on a graph $G=(V,E)$ is a probability measure $\mathbf{P}$ on $2^E$. 
We refer to elements $\go \in 2^E$ as \textbf{configurations}. 
We also say that an edge $e \in E$ is \textbf{present} (or \textbf{open}) in $\go$ if $e \in \go$, and say that $e$ is \textbf{absent} (or \textbf{closed}), otherwise.
The connected components of $\go$ are called \textbf{clusters}.
For $p \in [0,1]$, a bond percolation process on $G$ is called $\mathrm{Bernoulli}(p)$ percolation if every edge is present in a configuration independently with probability $p$. 
We denote the measure associated with $\mathrm{Bernoulli}(p)$ bond percolation by $\pr_p$ and omit the word ``bond''. 
A locally finite connected graph $G$ is called \textbf{unimodular} if its automorphism group $\Aut(G)$ is unimodular. 
Regardless of whether $G$ is unimodular or not, fix a (left) Haar measure $m$ on $\Gamma \coloneq \Aut(G)$, and define the \textbf{relative weight function} of a vertex $y$ with respect to a vertex $x$ as
\begin{equation}\label{def:haarweights}
\w \coloneq \w_\Gamma^x(y)\coloneqq m(\Gamma_y)/m(\Gamma_x),
\end{equation} 
where $\Gamma_v \coloneqq  \{\gamma\in\Gamma \mid \gamma v=v\}$ is the stabilizer of $v\in V$. 
Because Haar measure is unique up to a nonzero constant multiple, the definition of $\w$ does not depend on the choice of $m$; furthermore, $\w$ is a cocycle: $\w^x(y)\w^y(z) = \w^x(z)$.
By \cite[Theorem 1]{Trofimov}, $G$ is unimodular if and only if all $x, y \in V$ which are in the same $\Gamma$-orbit satisfy $\w^x(y) = 1$.
A set of vertices is called \textbf{light} (resp.\ ~\textbf{heavy}) if the sum of the weights of its elements is finite (resp.\ infinite). 
Since $\w$ is a cocycle, these
definitions do not depend on the reference point. 
In unimodular quasi-transitive graphs, $\w$ can take at most finitely many values, and thus every infinite cluster is automatically heavy. 
On the other hand, in the nonunimodular case, light clusters can be infinite. 
Moreover, \cite[Theorem 4.1.6]{Haggstrom99} implies that infinite light clusters cannot coexist with heavy ones in Bernoulli percolation, yielding four percolation phases for any quasi-transitive graph (see Figure~\ref{Fig:phases}).

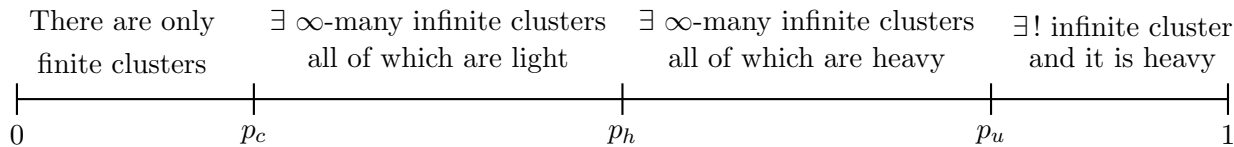
\begin{figure}[htb]
\begin{center}
\begin{tikzpicture}[thick, scale=0.7]

\draw [-] (-12,0) to (11,0);

\draw (-12,0.3) to (-12,-0.3);
\draw (-12,-0.3) node[below] {$0$};
\draw (-7.5,0.3) to (-7.5,-0.3);
\draw (-7.5,-0.3) node[below] {$p_c$};
\draw (-0.5,0.3) to (-0.5,-0.3);
\draw (-0.5,-0.3)node[below] {$p_h$};
\draw (6.5,0.3) to (6.5,-0.3);
\draw (6.5,-0.3) node[below] {$p_u$};
\draw (11,0.3) to (11,-0.3);
\draw (11,-0.3) node[below] {$1$};

\draw (-10,1) node[above] {There are only};
\draw (-10,0.3) node[above] {finite clusters};
\draw (-4,1) node[above] {$\exists$ $\infty$-many infinite clusters};
\draw (-4,0.3) node[above] {all of which are light};
\draw (3,1) node[above] {$\exists$ $\infty$-many infinite clusters};
\draw (3,0.3) node[above] {all of which are heavy };
\draw (9,1) node[above] {$\exists\,!$ infinite cluster};
\draw (9,0.3) node[above] {and it is heavy};
\end{tikzpicture}
\end{center}
\caption{The four phases of Bernoulli$(p)$ percolation, distinguished by the number of light/heavy infinite clusters.}
\label{Fig:phases}
\end{figure}

In 1999 H\"aggstr\"om, Peres, and Schonmann \cite{Haggstrom99} introduced the notion of \textbf{cluster repulsion}: \textit{a graph $G$ is said to exhibit cluster repulsion at $p\in[0,1]$ if for any two distinct infinite clusters $C,C'$, the set of vertices in $C$ within unit distance of $C'$ is finite $\pr_p$-almost surely}. 
They then asked the following question:

\begin{ques}[{\cite[p.~87]{Haggstrom99}}]\label{ques:HPS}
Does a quasi-transitive graph necessarily exhibit cluster repulsion for any value of $p\in[0,1]$?
\end{ques}

H\"aggstr\"om, Peres, and Schonmann remarked that for any graph $G$, the continuity of the so-called \textbf{two-point connectivity function} $\varphi_{u,v}(q)\coloneqq \pr_q(u\leftrightarrow v)$ at a value $p \in [0,1]$ for all pairs of vertices $u,v\in V$ implies cluster repulsion holds at $p$. 
Since for any graph $G$, the function $\varphi_{u,v}$ is monotone, there are at most countably many values of $p$ at which cluster repulsion fails (see \cref{cluster_repulsion_continuous}). 
However, Hutchcroft and Hermon \cite[Section 5.1]{hermon2021supercritical} pointed out that in general the two-point connectivity function $\varphi_{u,v}$ is not continuous at all values $p \in [0,1]$. In particular, for any nonamenable  (possibly nonunimodular) quasi-transitive graph $G$ with infinitely many infinite clusters at $p_u$, there exist vertices $u,v$ such that $\varphi_{u,v}$ has a jump discontinuity at $p_u$.
Examples of such graphs include Cayley graphs of Kazhdan groups \cite{LyonsSchramm99} and nonamenable products $G \times H$ \cite{Peres00}, such as $\mathrm{T}_k \times \bZ^d$ for $k \geq 3$.
Thus, proving cluster repulsion at $p_u$ for such graphs necessitates a different argument.
Perhaps surprisingly, it is still possible to leverage H\"aggstr\"om, Peres, and Schonmann's observation about the two-point connectivity function to prove cluster repulsion holds at all $p \in [0,1]$ in unimodular graphs (see \cref{thm:hf_touching}).

In 2006 Tim\'ar showed the positive answer to Question~\ref{ques:HPS} assuming that the graph is unimodular \cite{Timar06neighb}. Even though percolation on nonunimodular graphs has received considerable attention, particularly in recent years (see \cite{Timar06nonu, Pengfei18, Hutchcroft20, CTT22, HutchcroftPan24dim, HutchcroftPan24rel, wamen} and the references therein), Question~\ref{ques:HPS} still remains open. 
In this article, we make the first step towards the positive answer in the nonunimodular setting:

\begin{thm}\label{thm:heavyrepulsion}
Every quasi-transitive graph $G$ exhibits \textbf{heavy cluster repulsion} at each $p \in [0,1]$, i.e., for any two distinct infinite clusters $C,C'$, the set of vertices in $C$ within unit distance of $C'$ is light $\pr_p$-almost surely.
\end{thm}

Of course, the statement is nontrivial only for $p \in (p_h,p_u]$, when there are infinitely many heavy clusters (see \cite[Proposition 4.7]{Pengfei18} and \cite[Theorem 1.8]{wamen} for the  $p=p_h$ case) and is already established in the unimodular case in \cite{Timar06neighb}.

We give two proofs of \cref{thm:heavyrepulsion}, and we include both of them because each has its advantages. 
The first proof is a novel argument using hyperfiniteness, which relies on the monotonicity of the two-point connectivity function in $p$.
The second proof follows Tim\'ar's approach, and requires weighted analogues of several classical unimodular results, which are of independent interest.

Although our proofs work in a more general setting, where the relative weight function is defined with respect to a closed subgroup of $\Aut(G)$, we chose not to include this in the statement of \cref{thm:heavyrepulsion}. 
Moreover, two infinite clusters can share infinitely many boundary edges only if those edges lie in a single infinite $2$-connected block (maximal 2-connected induced subgraph of $G$), it is natural to restrict attention to such block and assume that $G$ is 2-connected. 
If such a component of $G$ contains vertices whose relative weights, as in \eqref{def:haarweights}, take only finitely many values, such a component must be unimodular and hence also falls into the scope of \cite{Timar06neighb}. Thus the main contribution of the present work is when $G$ is $2$-connected and nonunimodular.

While it is natural to expect that percolation results for infinite clusters in unimodular graphs to extend to their weighted analogs for heavy clusters in nonunimodular graphs, such generalizations often require new ideas (see, e.g., \cite{Pengfei18,CTT22,wamen}). 
Timár’s proof from \cite{Timar06neighb} and some of its ingredients are such examples.
Indeed, while our second proof of \cref{thm:heavyrepulsion} follows the overarching structure of Timár’s argument, the novelty lies in establishing weighted analogs of the local-global correspondence results used in his arguments.
Such a result is \cite[Theorem 7.2]{BLPS99inv}, which implies that if the expected degree of a random subforest $F$ of $G$ is greater than 2, then with positive probability, $F$ has a component with exponential growth.
To generalize Timár’s proof to the weighted setting, we would need to conclude weighted exponential growth, but no such result was available.

Recognizing that \cite[Theorem 7.2]{BLPS99inv} and other ingredients in Timár’s proof are really about \emph{nonhyperfiniteness}, we either circumvent these ingredients or prove generalizations of these results to the weighted setting by leveraging the connection with measured group/graph theory, via the cluster graphing construction introduced in \cite{Gaboriau05}, as well as the recent developments regarding amenability of measure-class-preserving Borel graphs.
Notably, one of our noteworthy results is the following extension of a theorem of Benjamini, Lyons, and Schramm \cite[Theorem 1.1]{BLS-perturb}, which is of independent interest. 
In fact, we prove a more general statement in \cref{thm:pos_Cheeger}.

\begin{thm}\label{thm:pos_Cheeger_Bernoulli}
Let $G$ be a connected locally finite quasi-transitive graph and let $\w$ be the relative weight function induced by a closed subgroup $\Gamma\subseteq\Aut(G)$ as in \eqref{def:haarweights}. 
Suppose $G$ is $\w$-nonamenable. 
Then for $\pr_p$-a.e.\ configuration $\go$ that contains a heavy cluster, there is a subgraph $\go' \subseteq \go$ such that $(\go, \go')$ is jointly $\Gamma$-invariant and $\Phi_V^{\w}(\go') > 0$.
\end{thm}

Here $\w$\textbf{-amenability} of a graph is defined with respect to the \textbf{weighted Cheeger constant} $\Phi_V^{\w}$, which was introduced in \cite[p. 17]{BLPS99inv} and further considered in \cite{wamen} (see \cite{kaimanovich1997amenability} for the analogous isoperimetric constant in the context of measured graph theory):

\begin{equation*}
\Phi^{\w}_V(G)\coloneqq \inf_{F\subseteq V\atop \abs{F}<\infty}\frac{\w^o(\partial_V F)}{\w^o(F)}.
\end{equation*}
Precisely, a graph is said to be $\w$-amenable if $\Phi^{\w}_V(G)$ = 0.

Finally, we mention an easy yet useful result on \textbf{relentless merging}, a notion related to cluster repulsion and also introduced in \cite{Haggstrom99}. 
Consider the standard monotone coupling of Bernoulli percolation, and denote the configuration at parameter $p$ by $\go_p$.

Relentless merging then says that \textit{a.s.\ for any $p_1<p_2$ in $(p_c,p_u)$ every infinite cluster in $\go_{p_2}$ contains infinitely many infinite clusters from $\go_{p_1}$.} 
This implies that all infinite clusters are born at the same time and then only merge as $p$ grows.
In \cite[Theorem 1.5]{Haggstrom99}, the authors prove that the relentless merging holds for unimodular graphs.
Their proof relies on \cite[Proposition 7.1]{Haggstrom99}, which says that in the regime with infinitely many infinite clusters, $\pr_p$-almost surely, every infinite cluster is at unit distance from infinitely many other infinite clusters. They then asked:

\begin{ques}[{\cite[p.~87]{Haggstrom99}}]\label{ques:merg}
Can the unimodularity assumption made in \cite[Theorem 1.5]{Haggstrom99} and \cite[Proposition 7.1]{Haggstrom99} be dropped? 
\end{ques}

It is not difficult to adapt the original arguments from \cite{Haggstrom99} to generalize these statements for heavy clusters. Since both of our arguments for \cref{thm:heavyrepulsion} rely on even stronger generalization of \cite[Proposition 7.1]{Haggstrom99} (see \cref{neighb_rel}) we record the following positive answer to \cref{ques:merg} in all cases besides the one where \textit{both} ${p_1}$ and ${p_2}$ are in $(p_c,p_h)$.

\begin{prop}\label{thm:rel_merg}
Let $G$ be a locally finite connected  quasi-transitive nonunimodular graph. 
Then under the monotone coupling, a.s.\ for any $p_1\in (p_c,p_u)$ and $p_2\in(p_h,1]$ such that $p_1 < p_2$, every infinite cluster in $\go_{p_2}$ contains infinitely many infinite clusters from $\go_{p_1}$.
\end{prop}

\begin{remark}\label{ExpectedChemicalDistance}
In \cite[Question 2.7]{Timar06neighb}, Timár asks whether for each transitive graph $G$ and $p > p_c(G)$, for every pair $x,y$ of neighboring vertices in $G$, the expected chemical distance $\E_p [d_\go(x,y) \mid x \leftrightarrow y]$ is finite.
By \textbf{chemical distance $d_\go$}, we mean the distance in the random subgraph $\go$, 
and by \textbf{expected chemical distance} $\E_p [d_\go(x,y) \mid x \leftrightarrow y]$, 
we mean the expected value of $d_\go(x,y)$ conditioned on $x,y$ being in the same cluster of the $\mathrm{Bernoulli}(p)$ configuration $\go$.
Tim\'ar observes that if this conjecture were true for all nonamenable graphs $G$ and $p \in (p_c(G), p_u(G))$, this would yield a substantially simpler proof of cluster repulsion for unimodular graphs.
However, as explained to us by Hutchcroft, a counterexample to this statement can be constructed with the additional assumptions of nonamenability and $p \in (p_c, p_u)$. 
Let $G$ be a nonamenable transitive simply connected planar graph with locally finite dual. 
By \cite[Theorem 6.1]{TomL2Boundedness}, for Bernoulli$(p_u(G))$ percolation on $G$ the expected chemical distance $\E_{p_u(G)}[d_\omega(x,y) \mid x \leftrightarrow y]$  is infinite between any pair $x,y$ of neighboring vertices in $G$.
Taking the free product of $G$ with an edge yields a graph $G' \coloneq G \ast e$ with $p_c(G') = p_c(G)$ and $p_u(G') = 1$. 
Thus, for $p = p_u(G) \in (p_c(G'),p_u(G'))$, the expected chemical distance between any two $G$-neighboring vertices in the same copy of $G$ in $G'$ is infinite.
\end{remark}

%%%%%%%%%%%%%%%%%%%%%%%%%%%%%%%%%%%%%
\subsection{Outline of the argument and our main contributions}
%%%%%%%%%%%%%%%%%%%%%%%%%%%%%%%%%%%%%

To establish cluster repulsion for unimodular transitive graphs, Tim\'ar \cite{Timar06neighb} uses the following three steps: 

\begin{enumerate}[(i), leftmargin=*]
\item \label{step:heavy=>multi-ended} If with positive probability, there are infinite \textbf{touching sets} (i.e., the set of points in a cluster that are within unit distance of another fixed cluster), then the insertion and deletion tolerance yields that there are touching sets which accumulate at more than two ends of $G$, and hence infinitely many ends of $G$.

\item \label{step:subforest} There is a random forest on the trifurcation vertices of these multi-ended touching sets with expected degree greater than $2$, and hence the forest is of exponential growth, relying on \cite[Theorem 7.2]{BLPS99inv}.

\item \label{step:mass_transport} Finally, Timár gives an ingenious mass transport argument, which uses insertion tolerance to create a bottleneck passage between two touching sets of exponential growth to overflow it, yielding a contradiction.
\end{enumerate}

We highlight that the central property of the forest in Tim\'ar's argument is nonhyperfiniteness, i.e.\ this forest cannot be written as an increasing union of invariant random graphs with finite components a.s. 
The large expected degree, $p_c<1$, and having more than two ends are all just properties that are equivalent to nonhyperfiniteness of random invariant forests in the unimodular setting \cite[Proposition 7.2]{BLPS99inv}. 
For nonunimodular graphs, however, these properties are not enough to conclude nonhyperfiniteness. 
In the absence of probabilistic tools to bypass this issue, both of our proofs of \cref{thm:heavyrepulsion} use a recent result \cite[Proposition 3.24]{CTT22} for locally finite measure class preserving Borel graphs.
We thus obtain the following stronger version of step \labelcref{step:heavy=>multi-ended} in the weighted setting (see \cref{thm:nonamenable}):

\begin{enumerate}[(i'), leftmargin=*]
\item \label{step:heavy=>nonhyperfinite} If with positive probability, there are heavy touching sets, then the invariant random partition into touching sets induced by the Bernoulli percolation clusters is nonhyperfinite.
\end{enumerate}

Since in the nonunimodular setting the number of ends does not detect hyperfiniteness, instead we work with $\w$-nonvanishing ends, which were recently introduced in \cite{CTT22,AnushRobin} (see \cref{sec:ends}).
Working with $\w$-nonvanishing ends raises several subtleties, which we address by proving a general lemma about an abstract symmetric neighboring relation of clusters (\cref{neighb_rel}), which \cite[Proposition 7.1]{Haggstrom99} and enables the proof of \cref{thm:rel_merg}.

Both of our proofs use step \labelcref{step:heavy=>nonhyperfinite}.
In our first proof, the novel insight is that the monotonicity of the two-point connectivity function in $p$, directly yields that the invariant random partition into touching sets induced by the Bernoulli($p$) percolation clusters is hyperfinite (\cref{thm:hf_touching}). 
This result is new even for unimodular graphs and allows us to completely bypass steps \labelcref{step:subforest,step:mass_transport}.

Our second proof replaces step \labelcref{step:subforest} by

\begin{enumerate}[resume*]
\item \label{step:Cheeger>0} There is random graph with positive weighted Cheeger constant (and hence has exponential growth of weights), whose components lie within touching sets.
\end{enumerate}

\noindent This is obtained from nonhyperfiniteness of the random partition into touching sets via a translation of Kaimanovich's theorem \cite{kaimanovich1997amenability} into the percolation setting (see Theorems~\ref{thm:Kai} and~\ref{thm:pos_Cheeger}).
As mentioned above, this step requires a completely different approach than in the unimodular case because results such as \cite[Theorem 7.2]{BLPS99inv} are not available in the nonunimodular setting, and it remains unclear whether a natural extension involving an appropriate notion of a weighted degree can even be developed \cite{Inselmann:MSc_thesis,poulin-anticost}.

Finally, we apply a weighted version of Timár's mass transport argument from step \labelcref{step:mass_transport} to the random graph from \labelcref{step:Cheeger>0}, which finishes the proof.

%%%%%%%%%%%%%%%%%%%%%%%%%%%%%%%%%%%%%
\subsubsection*{Declaration on the use of generative AI}
%%%%%%%%%%%%%%%%%%%%%%%%%%%%%%%%%%%%%

The authors did not use any AI tools to aid with research or any other preparation of this paper.
The first version of this paper was submitted to arXiv in September 2025.

%%%%%%%%%%%%%%%%%%%%%%%%%%%%%%%%%%%%%
\subsubsection*{Acknowledgments}
%%%%%%%%%%%%%%%%%%%%%%%%%%%%%%%%%%%%%

We would like to thank Louigi Addario-Berry, Ruiyuan (Ronnie) Chen, \'Ad\'am Tim\'ar, and Robin Tucker-Drob for many insightful conversations, some of which were not directly related to this project but nonetheless positively affected it. 
We also thank Tom Hutchcroft for explaining his results about expected chemical distance (see Remark~\ref{ExpectedChemicalDistance}).

S.B.~ and T.C.~were supported by NSERC CGS-M and FRQNT Bourse de formation à la maîtrise.
O.R.~was supported in part by an ISM Graduate scholarship.
G.T.~was supported in part by the RTG award grant (DMS-2134107) from the NSF and by the ERC Synergy Grant No. 810115 - DYNASNET.
A.Ts.~was supported by NSERC Discovery Grant RGPIN-2020-07120.

%%%%%%%%%%%%%%%%%%%%%%%%%%%%%%%%%%%%%%%%%%%%%%%%%%%%%%%%%%%%%%%%%%%%
%
\section{Preliminaries}\label{sec:prelim}
%
%%%%%%%%%%%%%%%%%%%%%%%%%%%%%%%%%%%%%%%%%%%%%%%%%%%%%%%%%%%%%%%%%%%%

%%%%%%%%%%%%%%%%%%%%%%%%%%%%%%%%%%%%%
\subsection{Insertion tolerance and mass transport}
%%%%%%%%%%%%%%%%%%%%%%%%%%%%%%%%%%%%%
We now recall essential percolation-theoretic definitions and tools and fix the notation for the remainder of the paper.

Given a set of configurations $A\subseteq 2^E$ and an edge $e\in E$, let $\Pi_eA=\{\go\cup \{e\}\mid \go\in A\}$ and $\Pi_{\neg e}A=\{\go\setminus \{e\}\mid \go\in A\}$. 
A percolation process $\mathbf{P}$ is called \textbf{insertion} (resp.\ \textbf{deletion})\textbf{ tolerant} if
$\mathbf{P}(\Pi_eA)>0$ (resp.\ $\mathbf{P}(\Pi_{\neg e}A)>0$) for every $e \in E$ and every non-null measurable set $A \subseteq 2^E$. Consider Bernoulli$(p)$ percolation on $G$. Then for every edge $e \in E$ and every measurable $A \subseteq 2^E$ we have
\begin{equation*}
\pr_p(\Pi_eA)\geq p\pr_p(A) \qquad\text{and}\qquad \pr_p(\Pi_{\neg e}A)\geq(1- p)\pr_p(A).
\end{equation*}
In particular, this implies that Bernoulli bond percolation is both insertion and deletion tolerant.

One of the main tools in the field is the Tilted Mass Transport Principle (TMTP). We state below a slightly more general version of TMTP than usual that is true for any action of a locally compact group with compact stabilizers on a countable set. For a different formulation that allows for non-uniform choice of a root, see \cite[Proposition 3.2]{Pengfei18}.

\begin{thm}[Tilted Mass Transport Principle]\label{thm:TMTP}
Let $\Gamma$ be a locally compact group acting quasi-transitively on a countable set $V$ with compact stabilizers, and let $\w_\Gamma$ be the $\Gamma$-invariant relative weight function on $V$ as in \eqref{def:haarweights}. Pick a complete set of representatives $\mvO = \{o_1, \cdots, o_L\}$ from each $\Gamma$-orbit and let $\rho$ be a random root uniformly sampled from $\mvO$. Then for any function $f:V^2 \to \bR^+$ that is invariant under the diagonal action of $\Gamma$, we have
\[
\E_{\rho}\left[\sum_{v\in V} f(\rho,v)\right]=\E_{\rho}\left[\sum_{v\in V} f(v,\rho)\w^{\rho}(v)\right].
\]
\end{thm}

\subsection{Invariant random partitions and touching classes}

Given a configuration $\go$ of a Bernoulli percolation and clusters $C,C'$ in $\go$, let $\tau(C,C')$ denote the \textbf{touching set} of $C$ with respect to $C'$, that is the set of points in $C$ that are within distance 1 of $C'$.
\textit{We highlight that the set of vertices in a cluster that neighbors another fixed cluster is not invariant, as it requires distinguishing the other cluster.} 
Instead, we must look at \textit{all} possible touching sets, and it turns out to be more convenient to work with a subsampled version which is disjoint by construction.
To make this precise we adopt the language of \textbf{invariant random partitions (IRP)} (also known as invariant random equivalence relations) introduced in \cite{RobinIRP,KechrisIRE}.

Let $\cP_V\subseteq$ $2^V\times 2^V$ be the set of all partitions of $V$. 
For a group $\Gamma$ acting on $V$, we say that a $\cP_V$-valued random process $\mathbf{\tau}$ is a $\Gamma$-invariant random partition of $V$ if the law of $\tau$ is a $\Gamma$-invariant probability measure on $\cP_V$.
For example, any invariant percolation process $\go$ on a graph $G=(V,E)$ induces an IRP via its \textbf{cluster--partition}, i.e.\ letting $x\sim y$ if $x,y$ are in the same cluster of $\go$.

Given a configuration $\go$ of a Bernoulli$(p)$ percolation we define an IRP $\wh{\mathbf{\tau}}$ as follows. Let each vertex $x$ pick a neighboring vertex $c(x)$ in $G$ that lies in a different $\go$-cluster uniformly at random (if such a vertex exists). 
If no such vertex exists, set $x$ to be a singleton in $\htau_{\go}$. \textit{Let a pair of vertices $\{x,y\}$ be in the same class of the partition $\htau_{\go}$ if and only if $x$ and $y$ are in the same $\go$-cluster and their respective picked vertices $c(x)$ and $c(y)$ are also in the same $\go$-cluster}. 

Note that $\wh{\mathbf{\tau}}$ is indeed an IRP, because $\pr_p$ is invariant and for each configuration the partition classes are defined in an invariant fashion. In fact the law of the pair $(\go,\htau_\go)$ is invariant. In particular, we will always think of $\htau_\go$ as a realization of a partition on top of a percolation configuration $\go$ and apply insertion/deletion techniques to the underlying $\go$, fixing the partition for all unaffected vertices. 

We will refer to the classes of $\htau_\go$ as \textbf{touching classes} to distinguish them from touching sets. Notice every class in the partition is uniquely labeled by an ordered pair of two clusters. When it is important to highlight classes indexed by different clusters, we denote the corresponding classes by $\htau_\go(C,C')$. Explicitly, $x \in \htau_\go(C,C')$ if $x \in C$ and $x$ picked $C'$.

%%%%%%%%%%%%%%%%%%%%%%%%%%%%%%%%%%%%%
\subsection{Ends of graphs}\label{sec:ends}
%%%%%%%%%%%%%%%%%%%%%%%%%%%%%%%%%%%%%

Let $G=(V,E)$ be a connected, locally finite graph. 
We say that a sequence of distinct vertices $(x_n)$ is \textbf{end-convergent} in $G$ if for any finite subgraph $F\subset G$ all but finitely many elements of the sequence are contained in the same connected component of $G\setminus F$. 
Two end-convergent sequences $(x_n)$ and $(y_n)$ are said to be \textbf{equivalent} if the sequence
\[
(z_n)=\begin{cases}
    x_n &n\mbox{ is odd}\\
    y_n &n\mbox{ is even}
\end{cases}
\] 
is end-convergent. An \textbf{end} of $G$ is an equivalence class of end-convergent sequences. We say that a set of vertices \textbf{converges to an end} $\xi$ if it belongs to the equivalence class $\xi$. 
For a subset of vertices $A$, and \textbf{end of $G$ along $A$} is an equivalence class of end-convergent sequences contained in $A$.

As we mentioned above, the existence of infinitely many ends is not as powerful of a property in the nonunimodular setting as it is in the unimodular one.
To address this problem a new notion of a weighted end was introduced in \cite{CTT22, AnushRobin}. See \cite[Section 5]{wamen} for discussion on various notions of weighted ends in applications to percolation theory.
Given a relative weight function $\w^o$ on $V$, as in \eqref{def:haarweights}, with respect to some reference point $o$, we call a set $A \subseteq V$ $\w$-\textbf{vanishing} if 
\[
\limsup_{x \in A} \w^o(x)=0;
\]
otherwise $A$ is $\w$-\textbf{nonvanishing}. Note that since $\w$ is a cocycle, this definition is stable under changes to  the reference point.
Finally, we call an end $\xi$ of $G$ $\w$-\textbf{vanishing} if any sequence $(x_n)$ that converges to it is $\w$-vanishing; otherwise $\xi$ is $\w$-\textbf{nonvanishing}.

\begin{comment}
\begin{remark}
Ends can also be defined using ultrafilters rather than the approach of end-convergent sets. 
Given a locally finite graph $G = (V,E)$, we say $A \subseteq V$ is \textbf{boundary-finite} if the set $\partial_EA$ of edges with one endpoint in $A$ and one in $A^c$ is finite. 
The set of all boundary-finite subsets $A$ forms a countable Boolean algebra when endowed with the usual set-theoretic operations of union, intersection, and complement. 
The \textbf{end compactification} $\hat{V}^G \supseteq V$ of $G$ is the corresponding Stone space of this Boolean algebra (see \cite{CTT22} for more detail on the ultrafilter approach). 
That is, $\hat{V}^G$ consists of the space of ultrafilters on this Boolean algebra, and vertices $x \in V$ correspond to principal ultrafilters in $\hat{V}^G$. 
The \textbf{endspace} of $G$ is $\partial V\coloneqq \hat{V}^G \setminus V$, the closed subspace of nonisolated points in $\hat{V}^G$.  
Given an end $\xi \in \partial V$, put
\[\hat{\w}(\xi) \coloneqq \limsup_{x \rightarrow \xi} \w(x) = \inf_{\hat{A} \ni \xi} \sup_{x \in A} \w(x) \in [0,\infty]\]
(where $\hat{A}$ ranges over clopen neighborhoods of $\xi$.) If $\hat{\w}(\xi) = 0$, we say that $\xi$ is $\w$-vanishing; otherwise $\xi$ is $\w$-nonvanishing. This definition is equivalent to the one in terms of geodesic rays stated above.
\end{remark}
\end{comment}

%%%%%%%%%%%%%%%%%%%%%%%%%%%%%%%%%%%%%
\subsection{Measured group theoretic results}\label{sec:prelim_mcp}
%%%%%%%%%%%%%%%%%%%%%%%%%%%%%%%%%%%%%

Measured group theory studies groups via their measurable actions on a standard probability space $(X,\mu)$. Hence it is closely connected to the theory of locally countable Borel graphs and countable Borel equivalence relations (CBERs), since these objects arise naturally as Schreier graphs and orbit equivalence relations of such actions, respectively. 
By the Feldman–Moore theorem \cite{Feldman-Moore}, in fact every locally countable Borel graph and every CBER arise in this fashion. 
CBERs are also in natural correspondence with Borel graphs. 
Indeed, every locally countable Borel graph is naturally associated with its connectedness relation, which is easily seen to be a CBER by Luzin--Novikov uniformization. 
Conversely, every CBER $\cR$ admits many \textbf{graphings}, i.e., Borel graphs $\cG$ whose connectedness relation is exactly $\cR$; for example, viewed as a graph, $\cR$ is its own (complete) graphing, but $\cR$ also admits locally finite graphings by \cite[remark after proof of Theorem 3.12]{JKLcber}.
%Hence measured group theory exploits interplays between these various points of view 

A CBER $\cR$ on $(X,\mu)$ is \textbf{probability measure preserving} (\textbf{pmp}) (resp.\ \textbf{measure class preserving} or \textbf{mcp}) if every Borel automorphism $\phi : X \to X$ with $\graph (\phi) \subseteq \cR$ preserves $\mu$ (resp.\ $\mu$-null sets). 
In fact, every CBER on $(X,\mu)$ is mcp on a $\mu$-conull set \cite[Proposition 2.1]{Miller:thesis}.
The lack of invariance of mcp CBERs is quantified similarly to the tilted mass transport principle (as in ~\cref{thm:TMTP}).
Indeed, every mcp CBER $\cR$ on $(X,\mu)$ admits an essentially unique function $\w_\mu : \cR \to \bR^+$, called the \textbf{Radon--Nikodym cocycle} (see \cite[Section 8]{KMtopics}), which satisfies the following (defining) properties, where we denote $\w_\mu^y(x) := \w_\mu(x,y)$:

\begin{enumerate}[label=(\roman*)]
\item (cocycle identity) $\w_\mu^y(x) \cdot \w_\mu^z(y) = \w_\mu^z(x)$ for all $\cR$-equivalent points $x,y,z \in X$;

\item (mass transport principle) for every $f : \cR \to [0, \infty]$, we have
\begin{equation*}
\int \sum_{z \in [x]_{\cR}} f(x,y) d\mu(x) = \int \sum_{y \in [x]_{\cR}} f(y,x) \w_\mu^x(y) d\mu(x).
\end{equation*}
\end{enumerate}

Similarly to the setting of quasi-transitive graphs, where unimodularity holds exactly when the relative weight function $\w \equiv 1$ in each orbit, % unimodularity holds if and only if the relative weight function $\w \equiv 1$ in each orbit, 
here the Radon--Nikodym cocycle $\w_\mu\equiv 1$ exactly when $\cR$ is pmp.

As for groups and graphs, amenability plays a central role in the study of CBERs.
A CBER $\cR$ on $(X,\mu)$ is said to be \textbf{$\mu$-amenable} if there is an $\cR$-invariant Borel conull set $X_0 \subseteq X$ and a sequence $(m^k)_{k\in\N}$ of Borel maps assigning to each $x\in X_0$ a probability vector $m^k_x \in \ell^1([x]_{\cR})$ such that $\lim_{k \to \infty} \|m^k_x-m^k_y\|_1 = 0$ for every $(x,y)\in \cR|_{X_0}$.
By the Connes--Feldman--Weiss theorem \cite[Theorem 10]{CFW}, this is equivalent to $\mu$-hyperfiniteness, another central notion in the theory, and we will use the two terms interchangeably.
A CBER is called $\mu$-\textbf{hyperfinite} if it is a countable increasing union of Borel subequivalence relations whose classes are finite $\mu$-a.e.
%A CBER $\cR$ is called \textbf{smooth} if it admits a Borel transversal (that is a Borel set $S\subseteq X$ that intersects every $\cR$-class in exactly one point).

A Borel graph $\cG$ on a standard probability space $(X,\mu)$ is called \textbf{pmp}/\textbf{mcp}/\textbf{$\mu$-amenable}/\textbf{$\mu$-hyperfinite} if its connectedness relation of $\cG$, denoted by $\cR_\cG$, is such.
Below and throughout, $\w_\mu$ will denote the Radon--Nikodym cocycle of $\cR_\cG$ with respect to $\mu$.

There are many different ways to detect $\mu$-amenability of a pmp graph. 
The most common approach, applicable only in the pmp setting, is via the theory of cost, introduced by Levitt \cite{Levitt:cost} and further developed by Gaboriau \cite{Gaboriau:mercuriale,Gaboriau:cost}. 
Cost is the analogue of the free rank for groups in the context of CBERs, and is defined as follows: for a pmp CBER $\cR$, its \textbf{cost} is the infimum of half the expected degree of its graphings. 
Unfortunately, these techniques do not appear to extend to the mcp or nonunimodular settings, as evidenced by \cite[Corollary 3.11]{JKLcber}, \cite{Inselmann:MSc_thesis}, and \cite[Theorem 4.1]{poulin-anticost}.

The successful approach to mcp theory, pioneered in \cite{AnushRobin,CTT22}, is to study the interplay between of the Radon--Nikodym cocycle and the geometry of the underlying mcp graphs. It has led to the development of the mcp framework, which yields structural results for mcp graphs \cite{AnushRobin,CTT22} and applications to percolation theory on nonunimodular transitive graphs \cite{CTT22,wamen}, as well as to the resolution of a long-standing question in the measure equivalence classification of Baumslag--Solitar groups \cite{MEBaumslagSolitar}.

In \cite{AnushRobin}, the authors prove that an mcp acyclic graph $\cG$ is $\mu$-hyperfinite if and only if $\mu$-a.e.\ $\cG$-component has at most two $\w_\mu$-nonvanishing ends, generalizing the classical Adams dichotomy \cite[Theorems 5.1 and 5.2]{Adams:trees_amenability} to the mcp setting. 
A direct corollary of this main result is a sufficient condition for nonamenability of mcp graphs:

\begin{prop}[{\cite[Proposition 3.24]{CTT22}}]\label{thm:Prop3.24_CTT}
Let $\cG$ be a locally finite mcp graph on a standard probability space $(X,\mu)$.
%and let $\w_\mu$ be the Radon--Nikodym cocycle of its connectedness relation with respect to $\mu$.
If a.e.~$\cG$-component has at least three $\w_\mu$-nonvanishing ends then $\cG$ is $\mu$-nowhere hyperfinite.
\end{prop}

In the forthcoming paper \cite{CTTTD}, \cref{thm:Prop3.24_CTT} is combined with \cite[Proposition 5.3]{Chen-Kechris} to establish the following relativization of \cref{thm:Prop3.24_CTT} to subrelations, which is crucially used in our proof of \cref{thm:heavyrepulsion}.
We provide a proof here for completeness and note that even in the pmp setting, this result does not appear in the literature.

\begin{thm}\label{thm:nonamenable}
Let $\cS \subseteq \cR$ be mcp CBERs on a standard probability space $(X,\mu)$.
Let $\cG$ be a locally finite graphing of $\cR$, and let $Y\subseteq X$ be the $\cS$-invariant set of all $x \in X$ such that $\cG$ has at least three $\w_\mu$-nonvanishing ends along $[x]_\cS$.
Then $\cS|_Y$ is $\mu$-nowhere hyperfinite.
\end{thm}

\begin{proof}
By \cite[Proposition 5.3]{Chen-Kechris}, there exist a CBER $\cE$ on a standard Borel space $Z$ and Borel maps $g: X \to Z$ and $f : Z \to X$ such that the following hold:
\begin{enumerate}[(i),leftmargin=*]
\item \label{item:embedding_S} $g$ is an injective \textbf{reduction of $\cS$ to $\cE$}, i.e., $x \cS y \iff g(x) \cE g(y)$ for all $x,y \in X$, and $g(X)$ is an $\cE$-complete section, i.e., $g(X)$ meets every $\cE$-class.

\item \label{item:class-bij} $f$ is a \textbf{class-bijective homomorphism from $\cE$ to $\cR$}, i.e., the restriction of $f$ to each $\cE$-class $[z]_\cE$ is a bijection onto the $\cR$-class $[f(z)]_\cR$.

\item \label{item:diagram_commutes} $f \circ g = \mathrm{id}_X$.
\end{enumerate}
Condition \labelcref{item:embedding_S} says that the measured equivalence relations $(\cS,X,\mu)$ and $(\cE|_{g(X)}, g(X), g_*\mu)$ are measure-isomorphic via $g$.
Thus, by Luzin--Novikov uniformization, there is a finite Borel measure $\nu$ on $Z$ whose restriction to $g(X)$ is $g_*\mu$ and $\cE$ is mcp with respect to $\nu$.
In particular, the Radon--Nikodym cocycle $\w_\nu$ of $\cE$ with respect to $\nu$ coincides with $\w_\mu$ on $g(X)$; more precisely, 
\begin{equation}\label{eq:cocycles_coincide}
 \w_\nu^{g(y)}(g(x)) = \w_\mu^y(x) \text{ for all } (x,y) \in \cS.
\end{equation}

By \labelcref{item:class-bij}, the $f$-pullback $\tilde\cG := (f \times f)^{-1}(\cG)$ of $\cG$ is a graphing of $\cE$ such that the $\tilde\cG$-component of each $z \in Z$ is isomorphic via $f$ to the $\cG$-component of $f(z)$.
It then follows from \labelcref{item:diagram_commutes} and \eqref{eq:cocycles_coincide} that for each $x \in X$, the $\tilde\cG$-component of $g(x)$ admits at least as many $\w_\nu$-nonvanishing ends along $g(X) \cap [g(x)]_\cE$ as the $\cG$-component of $x$ has $\w_\mu$-nonvanishing ends along $[x]_\cS$.
In particular, for each $y \in Y$, the $\tilde\cG$-component of $g(y)$ admits at least three $\w_\nu$-nonvanishing ends.

Thus, \cref{thm:Prop3.24_CTT} applies to $(\tilde\cG, Z, \nu/\nu(Z))$ and yields that $\tilde\cG$ is $\nu$-nowhere hyperfinite on the $\cE$-saturation of $g(Y)$.
In particular, $\cE|_{g(Y)}$ is $\nu$-nowhere hyperfinite, and hence $\cS|_Y$ is $\mu$-nowhere hyperfinite.
\end{proof}

We also recall a result of Kaimanovich \cite{kaimanovich1997amenability} that describes the relationship between the $\mu$-amenability of an mcp graph and the isoperimetric constants of its induced subgraphs. 
Our \cref{thm:pos_Cheeger} is a translation of Kaimanovich's result to the setting of percolation, which in turn enables the proof of our main result \cref{thm:heavyrepulsion} and a weighted generalization of \cite[Theorem 1.1]{BLS-perturb} as in \cref{thm:pos_Cheeger_Bernoulli}.
%We also recall a result of Kaimanovich \cite{kaimanovich1997amenability} that allows to witness $\mu$-amenability of an mcp graph via isoperimetric inequality.

\begin{defn}
Let $\cG$ be a locally finite mcp graph on a standard probability space $(X,\mu)$ and let $\w_\mu$ be the Radon--Nikodym cocycle of its connectedness relation with respect to $\mu$.
The \textbf{local Cheeger constant} at a vertex $x \in X$ is then defined as  
\begin{align*}
\Phi^{\w_\mu}_x(\cG)& \coloneqq  \inf\left\{ \dfrac{\w_\mu^x(\partial_{\cG} F)}{\w_\mu^x(F)}: F \subseteq [x]_{\cG}  \text{ finite and nonempty} \right\},
\end{align*}
where $\partial_{\cG}F$ denotes the outer-vertex boundary of $F$ in $\cG$.
\end{defn} 
Note that the map $x \mapsto \Phi^{\w_\mu}_x(\cG)$ is $\cG$-invariant, i.e., $\Phi^{\w_\mu}_x(\cG)=\Phi^{\w_\mu}_y(\cG)$ whenever $[x]_\cG = [y]_\cG$.

\begin{thm}[{\cite{kaimanovich1997amenability}}]\label{thm:Kai}
Let $\cG$ be an mcp graph on $(X,\mu)$ with uniformly bounded degree. Then $\cG$ is $\mu$-hyperfinite if and only if for any Borel set $B \subseteq X$ of positive measure, $\Phi_{x}^{\w_\mu}(\cG|_B) = 0$ for a.e.\ $x \in X$.
\end{thm}

We call $\Phi_x^{\w_\mu}(\cG)$ ``local" to distinguish it from a related isoperimetric constant that appears in several adaptations of the theorem above. For example, \cite{Elek12} and \cite{CGMTDoneended} adapt and extend this characterization in terms of a ``global" Cheeger constant. However, one can pass from one to another via the local-global bridge lemma \cite[Lemma 3.2]{Ts:hyperfinite_ergodic_subgraph} (see \cite[Theorem 2.31]{TasminThesis} for an implementation of this passage).

%%%%%%%%%%%%%%%%%%%%%%%%%%%%%%%%%%%%%
\subsection{The cluster graphing construction}\label{sec:cluster_graphing}
%%%%%%%%%%%%%%%%%%%%%%%%%%%%%%%%%%%%%

The general cluster graphing construction allows one to translate results from measured group theory to percolation theory, and vice versa.
In the percolation context, in the setting of unimodular graphs, the measured version of this construction was carried out in \cite[Sections 2.2, 2.3 and 5]{Gaboriau05}; an earlier example of the underlying abstract construction appeared in \cite[1.6.2]{Adams:trees_amenability}, and a generalization to all first-order structures in the Borel context was given in \cite[Section 3.2]{JKLcber}.

We use the treatment in \cite[Section 5B]{CTT22}, which is adapted to the nonunimodular setting.
Since that treatment is formulated for transitive graphs, whereas the present paper is concerned with quasi-transitive graphs, we make the necessary technical modifications and summarize them below.
As in \cite[Section 5B]{CTT22}, we do not restrict the construction to percolation processes on a graph $G$, but allow more general structures, such as invariant random partitions.

Let $\Gamma$ be a locally compact group acting quasi-transitively on a countable set $V$ with compact stabilizers, and let $\w_\Gamma$ be the $\Gamma$-invariant relative weight function on $V$ as in \eqref{def:haarweights}.
For instance, $\Gamma$ is a closed subgroup of $\Aut(G)$ for a locally finite connected graph $G$ on $V$.
Consider a free pmp action of $\Gamma$ on a standard probability space $(X,\mu)$. Denote the $\Gamma$-orbit equivalence relation on $X$ by ${E}_\Gamma$.
Choose a complete set of representatives $\mvO\coloneqq\{o_1, \cdots, o_L\} \in V$ from the $\Gamma$-orbits of $V$.
We extend the action of $\Gamma$ to $X \times V$, and consider the quotient $Y \coloneqq  (X \times V)/ \Gamma$ under the diagonal action of $\Gamma$, which is still a standard Borel space. 
Letting $\Gamma_{o_i}$ denote the stabilizer of $o_i$ under the action of $\Gamma$ on $V$, the space $Y$ can be naturally identified with the disjoint union $\sqcup_{i=1}^{L}X_{o_i}$ of standard Borel spaces $X_{o_i}\coloneqq X /\Gamma_{o_i}$. 

To define an appropriate measure on $Y$, we first define a surjection $\pi_{o_i}: X \rightarrow X_{o_i}$, for $1 \leq i \leq L$, by letting $\pi_{o_i}(x) = [(x,o_i)]_\Gamma$, which yields a probability measure $(\pi_{o_i})_* \mu$ on $X_{o_i}$.
We then obtain a measure $\nu$ on $Y$ as the convex combination 
\begin{equation*}
\nu\coloneqq \frac{1}{L}\sum_{i=1}^L (\pi_{{o_i}})_{\ast} \mu.    
\end{equation*}

We also define an equivalence relation $\cE_V$ on $(X \times V) / \Gamma$ by declaring $[(x,u)]_\Gamma$ and $[(y,v)]_\Gamma$ equivalent if $y = \gamma x$ for some $\gamma \in \Gamma$.
In particular, for each $x \in X$ and $u \in V$,
\[
\big[[(x,u)]_\Gamma\big]_{\cE_V} = \{[(x,v)]_\Gamma : v \in V\};
\]
that is, $\cE_V$ is a CBER.
Next for each $x \in X$ we define a bijection
\begin{align}\label{eq:phi_bijection}
\varphi_{x} : [[(x,u)]_\Gamma]_{\cE_V} \to V\quad \mbox{by}\quad[(x,u)]_\Gamma \mapsto u,
\end{align}
whose well-definedness follows from the freeness of $\Gamma \acts X$ because if $[(x,u)]_\Gamma = [(x,v)]_\Gamma$ then $u = v$.
Moreover, the map $x \mapsto \varphi_{x}$ is $\Gamma$-equivariant, that is: 
\begin{equation}\label{eq:phi_is_equivariant}
\varphi_{\gamma x} (\wt{y}) = \gamma \varphi_{x} (\wt{y}),
\end{equation}
for all $x \in X$, $\wt{y} \in [[(x,u)]_\Gamma]_{\cE_V}$, and $\gamma \in \Gamma$. See \cite[Section 5B]{CTT22} for more details.

Note that the action of $\Gamma$ on $Y$ may not preserve $\nu$. 
The Radon--Nikodym cocycle of the equivalence relation $\cE_V$ with respect to $\nu$, denoted by $\w_{\nu} : \cE_V \to \bR^+$, can be expressed in terms of the relative weight function $\w_\Gamma$ on $G$ derived from \eqref{def:haarweights}:
\begin{equation*}
\w_\nu([x,v]_\Gamma,[x,u]_\Gamma)\coloneqq\w_\nu^{[x,u]_\Gamma} ([x,v]_\Gamma)\coloneqq\w_\Gamma^{u}(v).
\end{equation*}

To see that, let $\E_{\rho}$ denote the expectation with respect to the uniform choice of a root from $\mvO$. 
Note that every Borel function $F:Y^2\to \bR^+$ can be rewritten as $\widehat{F} : V^2 \times X\to [0,\infty]$ defined as
\[
\widehat{F}(u,v;x) \coloneqq F\left([(x,u)]_\Gamma,[(x,v)]_\Gamma\right).
\]
Thus, by similar computations to those in the proof of \cite[Lemma 5.7]{CTT22} we get that 
\begin{align*}
\int_Y \sum_{\widetilde{y} \in [\widetilde{x}]} F(\widetilde{x},\widetilde{y}) d\nu(\widetilde{x}) 
&=\int_X \frac{1}{L}\sum_{i=1}^L\sum_{v \in V} F([x,o_i]_\Gamma,[x,v]_\Gamma) d\mu(x)=\E_{\rho}\sum_{v \in V} \int_X \widehat{F}(\rho,v;x) d\mu(x)\\
\textnormal{[by the TMTP]}\quad&=\E_{\rho}\sum_{v \in V} \int_X \widehat{F}(v,\rho;x) \w^{\rho}_\Gamma(v)d\mu(x)\\
&=\int_X \frac{1}{L}\sum_{i=1}^L \sum_{v \in V} F([x,v]_\Gamma,[x,o_i]_\Gamma) \w_{\nu}^{[x,o_i]_\Gamma}([x,v]_\Gamma)d\mu(x)\\
&=\int_Y \sum_{\widetilde{y}\in [\widetilde{x}]} F(\widetilde{y},\widetilde{x}) \w_\nu^{\widetilde{x}}(\widetilde{y})d\nu(\widetilde{x}).
\qedhere
\end{align*}

The following proposition is a restatement of \cite[Proposition 5.10]{CTT22} for quasi-transitive actions of $\Gamma$; since the same proof applies, we do not repeat it.

\begin{prop}[{Essentially \cite[Proposition 5.10]{CTT22}}]\label{graphing-and-rg}
In the above setting, we have:
\begin{enumerate}[(a)]
\item\label{graphing-to-rg} For every Borel graph $\cG \subseteq \cE_V$, there is a unique $\Gamma$-equivariant Borel map $\go_\cG : X \to 2^{V \times V}$ such that, for all $x \in X$, the map $\varphi_x$ from \eqref{eq:phi_bijection} is an isomorphism of relatively weighted graphs $(\cG |_{[\pi_{o_i}(x)]_{\cE_V}}, \w_{\nu})$ and $(\go_\cG(x), \w_\Gamma)$.
In particular, $(\go_\cG)_* \mu$ is the law of the $\Gamma$-invariant random graph $\go_\cG$.
\item\label{rg-to-graphing} For every $\Gamma$-equivariant Borel map $\go : X \to 2^{V \times V}$, there is a unique Borel graph $\cG_\go \subseteq \cE_V$ such that $\go_{\cG_\go} = \go$.
\end{enumerate}
\end{prop}

We now define cluster graphings for $\Gamma$-invariant random graphs on $V$.
Let $\mathbf{P}$ denote the law of a $\Gamma$-invariant random graph $\go$ on $V$.
A Borel graph $\cG$ on a standard probability space $(Y,\nu)$ is called a \textbf{cluster graphing} of $\mathbf{P}$ (or of $\go$) if, for some standard probability space $(X,\mu)$ and some free pmp action $\Gamma \acts (X,\mu)$, we have $Y = (X \times V) / \Gamma$, $\nu = \frac{1}{L}\sum_{i=1}^L (\pi_{{o_i}})_{\ast} \mu$, and $(\go_\cG)_* \mu = \mathbf{P}$.
As stated in \cite[Corollary 5.13]{CTT22}, every $\Gamma$-invariant random graph on $V$ admits a cluster graphing.

\begin{comment}
Suppose $(\go_n)_{n\in I}$ is a possibly infinite sequence of $\Gamma$-invariant random graphs on $V$ defined on the same probability space. Let $(\mathbf{P}_n)_{n\in I}$ denote the joint law for this sequence.
We say that a sequence $(\cG_n)_{n\in I}$  of Borel graphs on a standard probability space $(Y,\nu)$ is a \textbf{joining of cluster graphings} of $(\mathbf{P}_n)_{n\in I}$ (or of $(\go_n)_{n\in I}$) if for some standard probability space $(X,\mu)$ and some free pmp action $\Gamma \acts (X,\mu)$, we have $Y = (X \times V) / \Gamma$, $\nu = \frac{1}{L}\sum_{i=1}^L (\pi_{{o_i}})_{\ast} \mu$, and  $\mathbf{P}_n=(\go_{\cG_n})_* \mu$ for each $n\in I$.
\end{comment}

We say that a cluster graphing $\cG$ on a standard probability space $(Y,\nu)$ of a $\Gamma$-invariant random graph $\go$ on $V$ \textbf{factors} onto another $\Gamma$-invariant random process $\go'$ with law $\mathbf{P}'$ defined on the same probability space as $\omega$, if $(\cG,Y,\nu)$ is constructed from a free pmp action $\Gamma \acts (X,\mu)$ as above, which factors onto $(\cdot,\mathbf{P}')$.
For the purposes of this paper, the invariant random partition $\htau_\go$ of interest is always coupled with the underlying Bernoulli($p$) percolation process $\go$.
Thus, after enlarging the probability space $(2^E,\bP_p)$ with the extra randomness needed to produce the partition into touching classes, every cluster graphing of this enriched process also factors onto $\htau_\go$.

%%%%%%%%%%%%%%%%%%%%%%%%%%%%%%%%%%%%%%%%%%%%%%%%%%%%%%%%%%%%%%%%%%%%
%
\section{Hyperfiniteness of $\htau_\go$ in Bernoulli percolation}\label{sec:hf_tau}
%
%%%%%%%%%%%%%%%%%%%%%%%%%%%%%%%%%%%%%%%%%%%%%%%%%%%%%%%%%%%%%%%%%%%%

As in measured group theory, hyperfiniteness plays an important role in percolation theory on transitive graphs. 
This concept is implicitly used in \cite[Theorem 5.1]{BLPS99inv} and several subsequent results. 
It was notably generalized for invariant random graphs in \cite{URG} in the unimodular setting. 
\cite[Definition 4.1]{wamen} likely provides the first explicit definition of hyperfiniteness for an invariant random subgraph of an ambient graph $G$. Here, we formulate this concept more broadly for invariant random partitions, noting that the clustering of an invariant random subgraph of $G$ yields such a partition.

For partitions $\gs,\gt$ of a set $V$, we write $\gs \subseteq \gt$ to mean that $\gs$ is finer than $\gt$, i.e.\ each $\gt$-class is a union of $\gs$-classes.
For a sequence of partitions $(\gs_n)$ and a partition $\gs$, we write $\gs_n \nearrow \gs$ to mean that $(\gs_n)$ is increasing and $\gs=\bigcup_{n \in \N}\, \gs_n$.
We also say that $\gs$ is \textbf{class-finite} if each $\gs$-class is finite.

\begin{defn}[Hyperfiniteness for invariant random partitions]\label{def:hyperfin} 
Let $V$ be a countable set and let $\Gamma$ be a group acting on $V$. 
We say that a $\Gamma$-invariant random partition $\go$ of $V$ is \textbf{hyperfinite} if, on a potentially larger probability space, there exists an sequence $(\go_n)_{n\in\bN}$ of class-finite random partitions of $V$ such that the law of $(\go;(\go_n)_{n\in\bN})$ is $\Gamma$-invariant and $\go_n \nearrow \go$ a.s. 
\end{defn}

For a locally finite connected graph $G$ and a subgroup $\Gamma$ of $\Aut(G)$, we will call a $\Gamma$-invariant random subgraph $\go$ of $G$ \textbf{hyperfinite} if the induced IRP (i.e.,\ the cluster-partition of $\go$) is hyperfinite.
This recovers \cite[Definition 4.1]{wamen}.

\begin{remark}[The requirement of joint invariance]
It is perhaps worth noting that in the above definition, the condition about joint $\Gamma$-invariance does not follow from the $\Gamma$-invariance of laws of individual random variables.
Indeed, let $0<p<1$ and let $\go$ be the cluster partition of a Bernoulli($p$) percolation on the standard Cayley graph $G$ of the free group $\Gamma := \bF_2$ on generators $\{a,b\}$.
Then $\Gamma$ acts on $G$ by shift and the law of $\go$ is $\Gamma$-invariant.
Then the random partition $a \cdot \go$ obtained from $\go$ by shifting by the generator $a$ is also $\Gamma$-invariant,
however, the joint law of $(\go, a \cdot \go)$ is not $\Gamma$-invariant because a.s.\ $(b \cdot \go, b \cdot a \cdot \go) \ne (b \cdot \go, a \cdot b \cdot \go)$.
\end{remark}
\begin{remark}
It is often convenient to apply the notion of hyperfiniteness to invariant random partitions of invariant random subsets of the underlying set $V$.
For example, let $G$ be a locally finite connected graph with vertex set $V$, let $\Gamma$ be a closed subgroup of $\Aut(G)$, and let $\w$ be the induced relative weight function as in \eqref{def:haarweights}.
Then for a $\Gamma$-invariant random subgraph $\go$ of $G$, its $\w$-heavy clusters form a $\Gamma$-invariant random partition of a $\Gamma$-invariant random subset of $V$. 
We view such a partition as a partial partition of $V$ with a random domain.
To include this in the framework of (total) random partitions of $V$, we ``complete'' every partial partition to a total one by including singleton classes for vertices in $V$ outside of the domain.
Thus in the example of $\w$-heavy clusters, every vertex which belongs to a $\w$-light cluster will be its own class.
\end{remark}

The following results are needed in this work. They are essentially proven in the sources provided.

\begin{prop}\label{prop:hf_properties}
Let $\Gamma$ be a closed subgroup of $\Aut(G)$ that acts quasi-transitively on $G$ and let $\w$ be the induced relative weight function as in \eqref{def:haarweights}. Let $\go$ be a $\Gamma$-invariant random graph on $V$. Then

\begin{enumerate}[(a), leftmargin=*]

\item\label{prop:hf_properties_cl} $\go$ is hyperfinite if and only if $\go$ admits a cluster graphing on a standard probability space $(Y,\nu)$ that is $\nu$-hyperfinite. In fact, every cluster graphing of $\go$ that factors onto the hyperfinite exhaustion $(\go_n)_{n\in\bN}$ must be $\nu$-hyperfinite. See e.g.\ \cite[Lemma 4.3]{wamen}.

\item\label{prop:hf_properties_incr_union} The increasing union of jointly invariant hyperfinite random graphs $\go_n$ on $V$ is hyperfinite. See \cite[Proposition 4.2]{Feldman-Moore}.

\item\label{prop:hf_properties_acyclic} Assume that $\go$ is acyclic a.s. Then $\go$ is not hyperfinite if and only if it contains a connected component with $\ge 3$ $\w$-nonvanishing ends. See \cite[Theorem 5.3]{wamen} and \cite{AnushRobin}.
\end{enumerate}
\end{prop}

The result in the following lemma was already remarked in \cite{Haggstrom99}; we include it here with the proof for completeness.

\begin{lem}\label{cluster_repulsion_continuous}
    Let $G = (V,E)$ be a connected, locally finite, infinite graph. Suppose the pair connectivity function $\varphi_{u,v}: p \mapsto \pr_p(u \leftrightarrow v)$ is continuous at $p$ for all fixed $u,v \in V$. Then $G$ exhibits cluster repulsion at level $p$.
\end{lem}

\begin{proof}
    For any $\eps > 0$, given a configuration $\go_p$ sampled from $\pr_p$, we generate a configuration $\go_{p + \eps}$ with law $\pr_{p + \eps}$ by associating to each edge $e\in E$ a Bernoulli labels $B_e$ with success probability $\eps/(1-p)$, independently from each other and everything else, and setting
    \[
    \go_{p+\eps}(e)=\begin{cases}
        1 &\mbox{if }\go_p(e) = 1 \mbox{ or } (\go_p(e) = 0 \mbox{ and } B_e = 1),\\
        0 &\mbox{otherwise}.
    \end{cases}
    \]
    Denote the law of this coupling by $\pr$, so for any event $A \subseteq \{0,1\}^E$, we have $\pr_p(A) = \pr(\go_p \in A)$ and $\pr_{p+\eps}(A) = \pr(\go_{p+\eps} \in A)$.
    
    Suppose that with positive probability, there are two clusters in $\go_p$ that are within unit distance of each other infinitely often. Then there exist vertices $u,v \in V$ such that with positive probability, the $\go_p$-clusters that contain them are distinct and infinitely touch; call this event $A_{u,v}$. Clearly, $\{u \leftrightarrow v\} \subseteq A_{u,v}^c$.
    We claim that for all $\eps > 0$
    \[
    \pr(u \leftrightarrow v \text{ in } \go_{p+\eps} \text{ but not } \go_{p}) \geq \pr_p(A_{u,v}).
    \]
    Indeed, if $A_{u,v}$ occurs in $\go_p$, then $u$ is in the same $\go_{p+\eps}$-cluster as $v$ whenever at least one of infinitely many independent (of each other and everything else) Bernoulli$(\eps/(1-p))$ trials succeeds, which occurs almost surely by the Borel--Cantelli lemma. Hence, by our coupling,
    \[
    \pr_{p+\eps}(u \leftrightarrow v) \geq \pr_p(u \leftrightarrow v) + \pr_p(A_{u,v}).
    \]
    Thus, for all $\eps > 0$ we have that $\pr_{p+\eps}(u \leftrightarrow v) - \pr_p(u \leftrightarrow v) \geq \pr_p(A_{u,v}) > 0$. In particular, $\varphi_{u,v}$ is not continuous from above at $p$.
\end{proof}

\begin{cor}\label{cor:monotone}
    For any connected, locally finite graph $G=(V,E)$, cluster repulsion can only fail for countably many values of $p \in [0,1]$. 
\end{cor}
\begin{proof}
    This follows immediately from \cref{cluster_repulsion_continuous} and the observation that the map $p \mapsto \pr_p(u \leftrightarrow v)$ is monotone increasing.
\end{proof}

\begin{remark}
    As pointed out in the introduction, in general the two-point connectivity function $\varphi_{u,v}$ \textit{does} fail to be continuous at some value $p \in [0,1]$. Indeed, Hutchcroft and Hermon \cite[Section 5.1]{hermon2021supercritical} pointed out that for any nonamenable quasi-transitive (possibly nonunimodular) graph $G$ with infinitely many infinite clusters at $p_u$, there exist vertices $u,v$ such that $\varphi_{u,v}$ has a jump discontinuity at $p_u$. It remains open whether this is a quasi-transitive graph for which (standard) cluster repulsion fails at some value $p \in [0,1]$. Such an example would have to be nonunimodular, by \cite{Timar06neighb}.
\end{remark}

\begin{thm}\label{thm:hf_touching}
    Consider a jointly invariant random process $(\go,\htau_\go)$, where $\go$ is a configuration of a Bernoulli$(p)$ bond percolation and $\htau_\go$ is a partition of $\go$ into the touching classes. Then $\htau_\go$ is hyperfinite.
\end{thm}
\begin{proof}
For any $p\in(0,1)$, by \cref{cor:monotone} there must exist a monotone decreasing sequence $\eps_n\searrow 0$ such that a.s.\ no two clusters in Bernoulli$(p-\eps_n)$ percolation are within unit distance of each other infinitely often.

Consider the usual monotone coupling of Bernoulli percolation processes by assigning to each edge $e$ a random label $U_e\sim\mathrm{Uniform}[0,1]$ independently of each other and everything else. We may assume that $\go\sim\pr_p$ consists of all the edges whose labels $U_e<p$. Let $(\go_n)_{n\ge1}$ be the increasing sequence of configurations that consists of edges with labels $U_e<p-\eps_n$. In particular, $(\go,(\go_n)_{n\ge1})$ is invariant and $\go$ is the increasing union of $\go_n$. 

Recall that in the construction of $\htau_\go$ from $\go$ each vertex $x$ picked a neighboring vertex $c(x)$ in $G$ that lies in a different cluster uniformly at random (if such a vertex exists). Fix the choices made in $\go$. Using these choices we define a new sequence of IRPs $\wt{\tau}_n$ on $\go_n$ with finite classes whose increasing union will be $\htau_\go$.
First note that if $x$ and $c(x)$ lie in different clusters in $\go$ the same holds in $\go_n$ for all $n\ge1$. For every $n\ge1$ let $\wt{\tau}_n$ be an IRP where a pair of vertices $\{x,y\}$ belongs to the same class if $x$ and $y$ are in the same $\go_n$-cluster and $c(x),c(y)$ are also in the same $\go_n$-cluster. Clearly $\wt{\tau}_n$ is a finer partition than $\htau_{\go_n}$, as there could be a vertex in $\go$ that does not admit a neighbor from a different cluster in $\go$ (and thus is a singleton in $\htau_\go$ and in $\wt{\tau}_n$) but does in $\go_n$ and could belong to a non-singleton class in $\htau_{\go_n}$. Clearly the law of the joint process $(\go,\htau_\go,(\go_n,\wt{\tau}_n)_{n\ge1})$ is still invariant and $\go_n\nearrow \go$, $\wt{\tau}_n\nearrow \htau_\go$. From the cluster repulsion in $\go_n$ for every $n\ge1$ it follows that the classes of $\htau_{\go_n}$, and thus of  $\wt{\tau}_n$, are finite (and thus light) a.s. In particular, $\htau_\go$ is hyperfinite.

\end{proof}

%%%%%%%%%%%%%%%%%%%%%%%%%%%%%%%%%%%%%%%%%%%%%%%%%%%%%%%%%%%%%%%%%%%%
%
\section{Heaviness of $\htau_\go$ yields its nonhyperfiniteness}\label{sec:proof_of_main}
%
%%%%%%%%%%%%%%%%%%%%%%%%%%%%%%%%%%%%%%%%%%%%%%%%%%%%%%%%%%%%%%%%%%%%

The goal of this section is to prove the following two statements.
The first proposition is an analogue of \cite[Lemma 2.3]{Timar06neighb}.

\begin{prop}\label{infinitely many nonvanishing ends}
Consider a jointly invariant random process $(\go,\htau_\go)$, where $\go$ is a configuration of a Bernoulli$(p)$ bond percolation, and $\htau_\go$ is a partition of $\go$ into the touching classes. 
Suppose the conclusion of \cref{thm:heavyrepulsion} does not hold. Then with positive probability there is a cluster that has infinitely many $\w$-nonvanishing ends along some $\htau_\go$-class, none of which are isolated.
\end{prop}

The second statement shows that infinitely many $\w$-nonvanishing ends is sufficient to conclude that the invariant random partition $\htau_\omega$ is not hyperfinite.

\begin{lem}\label{non_hf_class}
Consider a jointly invariant random process $(\go,\htau_\go)$, where $\go$ is a configuration of a Bernoulli$(p)$ bond percolation and $\htau_\go$ is a partition of $\go$ into the touching classes. Assume that a.s.\ there is a cluster of $\go$
that admits a touching class in $\htau_\go$ along which it has $\ge 3$ $\w$-nonvanishing ends. Then $\htau_\go$ is not hyperfinite.
\end{lem}

Note that in light of Theorem~\ref{thm:hf_touching}, \cref{infinitely many nonvanishing ends} and \cref{non_hf_class} immediately imply Theorem~\ref{thm:heavyrepulsion}.

\begin{proof}[First proof of~\cref{thm:heavyrepulsion}]
 Consider a jointly invariant random process $(\go,\htau_\go)$, where $\go$ is a configuration of a Bernoulli$(p)$ bond percolation and $\htau_\go$ is a partition of $\go$ into the touching classes. By \cref{thm:hf_touching}, $\htau_\go$ is hyperfinite. Suppose towards contradiction that the conclusion of \cref{thm:heavyrepulsion} did not hold. Then by \cref{infinitely many nonvanishing ends}, with positive probability some cluster of $\omega$ has infinitely many $\w$-nonvanishing ends along some $\htau_\omega$-class. By \cref{non_hf_class}, this implies that $\htau_\omega$ is not hyperfinite, a contradiction.
\end{proof}

%%%%%%%%%%%%%%%%%%%%%%%%%%%%%%%%%%%%%
\subsection{Neighboring relations and relentless merging}
%%%%%%%%%%%%%%%%%%%%%%%%%%%%%%%%%%%%%
We begin by working towards \cref{infinitely many nonvanishing ends}. A key step in the argument that establishes relentless merging of infinite clusters of Bernoulli$(p)$ percolation on unimodular transitive graphs is \cite[Proposition 4.7.1]{Haggstrom99}.  It shows that if $p$ is such that there are infinitely many infinite clusters, then $\pr_p$-a.s.\ every infinite cluster is at unit distance from infinitely many other infinite clusters. In this subsection, we generalize this statement to a general invariant neighboring property of clusters; it can then be applied to properties such as having a nonempty touching set (recovering the original statement), having a heavy touching set, finitely touching, infinite-lightly touching, and so on.

\begin{defn}
Let $\cN$ be an invariant property of a subset of vertices of $G$.
We say that a cluster $C\subset \go$ is an $\cN$\textbf{-neighbor} of a cluster $C'\subset \go$, and write $C\sim_{\cN} C'$, if the set $\tau(C,C')$ of vertices in $C$ that are within unit distance of $C'$ satisfies the property $\cN$. We say that the property is \textbf{symmetric} %for clusters 
if %for clusters
$C\sim_{\cN} C'$ always implies that $C'\sim_{\cN} C$. 
\end{defn}
\begin{ex}
Consider the properties listed above: being nonempty, heavy, finite, infinite-light. Recall that the relative weight function from \eqref{def:haarweights} can be rewritten as 
\[
\w^x(y):=m(\Gamma_y)/m(\Gamma_x)={\abs{\Gamma_y x}}/{\abs{\Gamma_x y}},
\]
where $\Gamma_v := \{\gamma\in\Gamma \mid \gamma v=v\}$ is the stabilizer of $v\in V$. Hence, if $d$ is the maximum degree in $G$, for two adjacent vertices $x$ and $y$ it always holds that 
\[
\frac{1}{d} \le \w^x(y)\le d .
\]
This implies that all of these properties are symmetric. On the other hand, a property such as ``finite touching with exactly $k$ elements" is not symmetric.
\end{ex}
\begin{remark}
In this paper we focus on the symmetric properties. However most of the arguments in the proof of \cref{neighb_rel} extend beyond this assumption. Moreover, it could be extended to different ``neighboring properties," for instance to those defined on clusters at distances greater than 1.
\end{remark}

\begin{lem}\label{neighb_rel} 
For any symmetric invariant property $\cN$ on $G$, for $\pr_p$-a.e.\ configuration $\go$, either every heavy cluster in $\go$ has no $\cN$-neighboring clusters or every heavy cluster in $\go$ has infinitely many $\cN$-neighboring clusters.
\end{lem}
\begin{proof}
Let $\cN$ be an invariant property and let $H_{\cN}$ be the corresponding graph, that is, a graph whose vertices are the clusters of $G$, with edges between every pair of $\cN$-neighbors.
By the indistinguishability of heavy clusters of Bernoulli$(p)$ percolation \cite[Theorem 1.1]{Pengfei18} $H_{\cN}$ is $d$-regular a.s., for some possibly infinite $d.$ Suppose $d\in[1,\infty)$, then $H_{\cN}$ must contain only infinite components a.s. Indeed, by insertion tolerance, given a finite component in $H_{\cN}$, insert edges (in the underlying percolation configuration $\go$) between neighboring clusters. This yields that with positive probability $H_{\cN}$ contains an isolated vertex, but inserting these edges did not affect any other cluster, contradicting the $d$-regularity of $H_{\cN}$. In particular, this implies that $d\neq 1$. By a similar argument, there cannot be more than one infinite component in $H_{\cN}$.

If $d\in ( 2,\infty)$ then any pair of $\cN$-neighbors must have exactly $d-2$ common $\cN$-neighbors and each of them has one more $\cN$-neighbor. Otherwise placing an edge between them increases the degree in $H_{\cN}$. Consider such a pair $C_1$ and $C_2$ and denote one of their common $\cN$-neighbors as $C_3$, while we let $D_1$ (resp.\ $D_2$) be a $\cN$-neighbor of $C_1$, but not $C_2$ (resp.\ $C_2$, but not $C_1$). 
Note that $C_3$ has to be an $\cN$-neighbor of exactly one of $D_1$ or $D_2$ 
(inserting an edge between $C_1$ and $C_2$ forces there to be a unique cluster $\cN$-neighboring $C_1\cup C_2$ but not $C_3$)
, say $D_2$.
Now inserting an edge between $C_1$ and $C_3$ yields a cluster that is a $\cN$-neighbor with $C_2$ and every single one of its $\cN$-neighbors, a contradiction. 

Thus, still assuming that $d<\infty$, we have that $d=2$ and every single connected component of $H_{\cN}$ is isomorphic to a $\bZ$-line. We first claim that there can be at most one such component. Otherwise, there must be a path in $\go$ that connects two clusters from two different such infinite lines of $H_{\cN}$ and does not intersect any other clusters in them. Inserting such a path affects at most finitely many clusters and increases the degree in $H_{\cN}$ of at least two two of them.

Suppose $H_{\cN}$ is isomorphic to a single $\bZ$-line. 
Notice that again by insertion tolerance there cannot be a path connecting any heavy cluster to a non $\cN$-neighboring heavy cluster that does not intersect its $\cN$-neighbors.
Indeed, inserting such a path increases the degree in $H_{\cN}$. 
Let heavy clusters $C_i$, $i\in\{1,\ldots,5\}$, be such that $C_i$ and $C_{i+1}$ are $\cN$-neighboring clusters. 
Pick a path $\pi$ in $\go$ connecting $C_2$ and $C_4$, and delete the $1$-ball around this path. 
This splits clusters $C_i$, $i\in\{1,2,3\}$, into finitely many clusters. 
However, some of them still must be heavy and (by the argument above) after this modification $H_{\cN}$ still must be isomorphic to a single $\bZ$-line. 
That is, there are clusters $D_i\subset C_2\cup C_3\cup C_4$, $i\in\{1,\ldots,k\}$, $k\ge 3$, such that all of the following hold:
\begin{enumerate}
\item $D_1\subset C_2$ is a $\cN$-neighbor of $C_1$,
\item $D_k\subset C_4$ is a $\cN$-neighbor of $C_5$,
\item there is a heavy cluster $D_\ell\subset C_3$,
\item $D_i$ and $D_{i+1}$ are $\cN$-neighboring clusters for all $i \in \{1, \cdots, k - 1\}$.
\end{enumerate}

On such an event, insert $\pi$ and insert edges connecting $\pi$ to $D_1$ and $D_k$ (such edges exist because each $D_i$ intersects the $2$-ball around $\pi$).
After this modification, the degree (in $H_{\cN}$) of the cluster $C'$ that contains $D_1, \pi,$ and $D_k$ is at least three, as $C'$ has an edge to $C_1$, an edge to $C_5$, and an edge to the component of $H_{\cN}$ that contains $D_\ell$ (see \cref{fig:d=2 line}). Contradiction, thus $d\in\{0,\infty\}$.
\tikzset{
big dot/.style={
circle, inner sep=0pt, 
minimum size=1.5mm, fill=black
}
}
\begin{figure}[htb]
\begin{center}
\begin{tikzpicture}[thick, scale=0.6]

\draw (-11,0) -- (-10,0) -- (-10,8) -- (-11,8);
\draw (-10.5,0) node[below] {$C_1$};
\draw (-8,0) rectangle (-4,8); 
\draw (-6,0) node[below] {$C_2$};
\draw (-2,0) rectangle (2,8);
\draw (0,0) node[below] {$C_3$};
\draw (4,0) rectangle (8,8);
\draw (6,0) node[below] {$C_4$};
\draw (11,0) -- (10,0) -- (10,8) -- (11,8);
\draw (10.5,0) node[below] {$C_5$};
\draw [<->,line width=0.7mm] (-9.9,4) to node[below] {heavy} (-8.1,4);
\draw [<->,line width=0.7mm] (-3.9,4) to node[below] {heavy} (-2.1,4);
\draw [<->,line width=0.7mm] (8.1,4) to node[below] {heavy} (9.9,4);
\draw [<->,line width=0.7mm] (2.1,4) to node[below] {heavy} (3.9,4);

\draw[dotted] (-7,5.2) rectangle (7,6.8);
\draw plot [smooth]  coordinates {(-6,5.5) (-4,6.5) (-1,5.5) (1,5.6) (4,6.2) (6,5.5)};
\draw (0,5.7) node[above] {$\pi$};
\node[big dot] (x1) at (-6,5.5) {};
\node[big dot] (x2) at (6,5.5) {};
\node[big dot] (y1) at (-7.1,4.5) {};
\node[big dot] (y2) at (7.1,4.5) {};
\draw [blue] (x1) to (y1);
\draw [blue] (x2) to (y2);

\draw[dashed] (-7.9,0.1) rectangle (-6.5,4.9);
\draw (-7.2,0.5) node[above] {$D_1$};
\draw[dashed] (-5.4,0.1) rectangle (-4.1,4.9);
\draw (-4.7,0.5) node[above] {$D_2$};
\draw [<->,line width=0.7mm] (-6.4,4) to (-5.5 ,4);
\draw[dashed] (-1.8,0.1) rectangle (1.8,4.9);
\draw (0,0.5) node[above] {$D_3$};
\draw[dashed] (7.9,0.1) rectangle (6.5,4.9);
\draw (7.2,0.5) node[above] {$D_5$};
\draw[dashed] (5.4,0.1) rectangle (4.1,4.9);
\draw (4.7,0.5) node[above] {$D_4$};
\draw [<->,line width=0.7mm] (6.4,4) to (5.5,4);
\end{tikzpicture}	
\end{center}
\caption{Heavy clusters $C_1,C_2,\dots,C_5$ are depicted by solid rectangles, the 1-neighborhood $B_1(\pi)$ of a path $\pi$ is depicted by the dotted rectangle, while dashed rectangles depict the acquired heavy clusters after the deletion of $B_1(\pi)$ (in this case $k=5$). The bold double arrows depict a heavy touching relation. Finally, the blue edges are the edges inserted at the end of the argument to connect $D_1$ and $D_k$, yielding a heavy cluster with more than $2$ heavy neighbors.}\label{fig:d=2 line}
\end{figure}
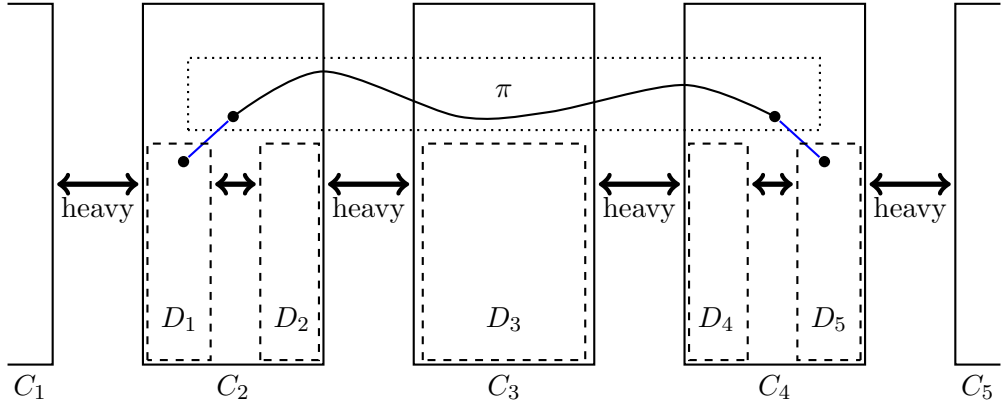
\end{proof}

\begin{cor}\label{inf touching}For $p\in (p_h,p_u),$ every infinite cluster is within distance one of infinitely many other infinite clusters $\pr_p$-a.s.
\end{cor}
\begin{proof}
Let $p\in (p_h,p_u).$ By insertion tolerance, with positive probability there are two $p$-clusters within distance one of each other. Since having distance one is a symmetric invariant property, it follows from \cref{neighb_rel} that $\pr_p$-a.e.\ infinite cluster is within distance one of infinitely many other infinite clusters. 
\end{proof}
Let $\wt{\mathbf{P}}$ denote the probability measure that assign an independent Uniform$[0,1]$ random variable $U_e$ to each edge $e$ in $G$. As noted above, deleting all edges with labels above a threshold $p$ yields a monotone coupling of Bernoulli$(p)$ percolation processes on $G$.
We denote the corresponding configurations by $\go_p$ and call a cluster $C\in\go_p$ a \textbf{$p$-cluster}.

\begin{proof}[Proof of~\cref{thm:rel_merg}]
If $p_h(G)=p_u(G)$ the statement is trivial, so we may assume that it is not the case.
Suppose $p_1\in (p_c, p_h]$, then $\wt{\mathbf{P}}$-a.s.\ all $p_1$-clusters are light (since we assumed that $p_h<p_u$ the absence of heavy clusters at level $p=p_h$ follows from \cite[Corollary 6.3]{Pengfei18} or \cite[Theorem 1.8]{wamen}). We argue by contradiction. Suppose that w.p.p.~there is an infinite $p_2$-cluster $C$ that contains only finitely many infinite $p_1$-clusters. Note that $C$ is heavy and it contains at least one infinite $p_1$-cluster by \cite[Theorem 4.1.3]{Haggstrom99}. On this event, since all $p_1$-clusters are light, there is a finite set of vertices $A(p_1,p_2)\subset C$ that consists of points in $C$ that were of maximal weight in some infinite $p_1$-cluster. This violates the TMTP.

The proof of the remaining statement is the same as in \cite[Theorem 4.1.4]{Haggstrom99}. The only place where the assumption of the unimodularity of $G$ is used is when \cite[Proposition 7.1]{Haggstrom99} is applied, which we now replace by \cref{inf touching}.
\end{proof}

%%%%%%%%%%%%%%%%%%%%%%%%%%%%%%%%%%%%%
\subsection{Number of $\w$-nonvanishing ends along $\htau$-classes}\label{sec:}
%%%%%%%%%%%%%%%%%%%%%%%%%%%%%%%%%%%%%
We start with the following basic ``subsampling" lemma.

\begin{lem}\label{subsamplig}
Let $A=\{a_n\}_{n\in\bN}$ be a set equipped with a weight function $\w:A\to\bR$ such that $\w(A)=\infty$. To each element $a_n$ associate an independent Bernoulli$(p_n)$ random variable $X_n$, with $p_n>c$ for some constant $c>0$. Let $B\subseteq A$ consist of elements $a_n$ for $n$ such that $X_n=1$. Then a.s.\ $\w(B)=\infty$.
\end{lem}
\begin{proof}
It is enough to show that $\limsup_{n\to\infty} \w(a_n)\ind_{a_n\in B}>0.$
Observe that $\w(B)$ stochastically dominates the random variable $W\coloneqq \sum_{n\in\mathbb{N}}Y_n\w(a_n),$ where $\{Y_n\}_{n\in\N}$ is a sequence of independent Bernoulli$(c)$ random variables.

Suppose $\limsup_{n\rightarrow\infty} \w(a_n)=0$, since otherwise we are done. Let $N_0 \coloneqq \inf\left\{n\ge 1: \w(a_n) \le 1 \right\}$. For $N\ge N_0$, let $W_N\coloneqq  \sum_{n=N_0}^N Y_n \w(a_n)$ be a truncated version of $W$. Then $\E[W_N] = c\sum_{n=N_0}^N \w(a_n)$ and $
\var(W_N) = c(1-c)\sum_{n=N_0}^N \w(a_n)^2 .$ 
We now show that $\lim_{N\rightarrow\infty}\pr(W_N < \frac{1}{2}\E [W_N])= 0$ and hence that $W$ is infinite almost surely. Indeed, by Chebyshev's inequality,
\[
\pr\left(W_N < \frac{1}{2}\E [W_N]\right) = \pr\left(\E [W_N]-W_N >\frac{1}{2}\E (W_N)\right)\leq  4\frac{\var(W_N)}{\E [W_N]^2} .%= 4\frac{c\sum_{n=N_0}^N \w(a_n)}{\left(c(1-c)\sum_{n=N_0}^N \w(a_n)^2\right)^2}
\]
Since $\w(a_n)\ge \w(a_n)^2$ for $n\ge N_0$, it follows that
\[
\lim_{N\rightarrow\infty}\pr\left(W_N <  \frac{1}{2}\E (W_N)\right) \leq \lim_{N\rightarrow\infty}\frac{4(1-c)}{c\sum_{n=N_0}^N\w(a_n)}=0.
\]
\end{proof}

\begin{lem}\label{htau}
Consider a jointly invariant random process $(\go,\htau_\go)$, where $\go$ is a configuration of a Bernoulli$(p)$ bond percolation and $\htau_\go$ is a partition of $\go$ into the touching classes. 
Then a.s.\ for any two distinct clusters $C,C'$ in $\go$, if  $\tau(C,C')$ is heavy then $C$ has at least one $\w$-nonvanishing end along $\htau_\go(C,C')$.
\end{lem}
\begin{proof}
By the local finiteness of $G$, each vertex in $C$ that neighbors $C'$ belongs to $\htau(C,C')$ with probability greater or equal to $1/d_{\max}$. If the neighboring set $\tau(C,C')$ was originally heavy, then it follows from  \cref{subsamplig} that $\htau_\go(C,C')$ is also heavy. 

Suppose that with positive probability there are two heavy clusters $C$ and $C'$ such that $\htau(C,C')$ is heavy but does not have a $\w$-nonvanishing end in $C$.
Then it must contain finitely many vertices of maximal weight. Now, given $\htau_{\go}$ let every vertex $x\in V$ distribute unit mass between all vertices of maximal weight in $[x]_{\htau_{\go}}$; if there are no vertices of maximal weight $x$ sends mass one to itself. Clearly, the expected sent out mass is finite, while on the event above the expected mass received is infinite. This contradicts TMTP.
\end{proof}

The following proposition is an analogue of \cite[Lemma 2.3]{Timar06neighb}.

\begin{prop}\label{infinitely many nonvanishing ends}
Consider a jointly invariant random process $(\go,\htau_\go)$, where $\go$ is a configuration of a Bernoulli$(p)$ bond percolation, and $\htau_\go$ is a partition of $\go$ into the touching classes. 
Suppose the conclusion of \cref{thm:heavyrepulsion} does not hold. Then with positive probability there is a cluster that has infinitely many $\w$-nonvanishing ends along some $\htau_\go$-class, none of which are isolated.
\end{prop}
\begin{proof}
Let $A$ be the event that there exist heavily neighboring clusters. Observe that $A$ is insertion and deletion tolerant. We assume that the conclusion of \cref{thm:heavyrepulsion} does not hold, i.e.\ that $A$ occurs with positive probability. 
We will prove that with positive probability, there are two heavy clusters $C$ and $C'$ such that the corresponding touching class $\htau_\go(C,C')$ has $\ge3$ $\w$-nonvanishing ends. The conclusion then follows from a standard application of TMTP analogous to \cite[Proposition 3.9]{LyonsSchramm99}. 
We proceed by case analysis.

\medskip

\noindent\textbf{Case 1.} First, suppose that with positive probability, there are four clusters $C_1,C_2,D_1,D_2$ such that $$\w(\htau_\go(C_1,D_1))=\w(\htau_\go(C_2,D_2))=\infty$$ and the remaining touching \textit{sets} (which contain corresponding touching classes) between these clusters are empty.
From such a quadruple of clusters one can show, in the manner of the proof of \cite[Lemma 2.3]{Timar06neighb}, that with positive probability there must be two heavy clusters $C$ and $C'$ such that the corresponding touching class $\htau_\go(C,C')$ has $\ge3$ $\w$-nonvanishing ends. 

\medskip

\noindent\textbf{Case 2.} Next, suppose that a.s.\ no such clusters $C_1,C_2,D_1,D_2$ exist. 
Let $\cN$ represent the heavy neighboring relation between clusters, which is equivalent to the corresponding class in $\htau_\go$ being heavy (by \cref{htau}). By \cref{neighb_rel}, we know that any heavy cluster $C$ has infinitely many such $\cN$-neighbors. Say three of them are called $C_1$, $C_2$, and $C_3$. 

\medskip

\noindent\textbf{Case 2a.} Suppose that $C_1$ and $C_2$ share all their $\mathcal{N}$-neighbors.
Choose two common $\cN$-neighbors of $C_1$ and $C_2$ (neither of which is $C$), and call them $D$ and $D'$. 
Insert edges between $C$ and $C_1$, $C_1$ and $D$, and $C_1$ and $D'$ to create a new cluster $C'$ in $\go$. Then by \cref{htau}, $C_2$ has $\ge 3$ $\w$-nonvanishing ends along  $\htau_\go(C_2,C')$ with positive probability. The cases where $C_1$ and $C_3$, or $C_2$ and $C_3$, share all their $\cN$-neighbors are handled analogously.

\medskip

\noindent\textbf{Case 2b.} The other possibility is that  there exist clusters $D_i, i\in\{1,2,3\}$ such that $D_i$ is an $\cN$-neighbor of $C_i$, but is not a $\cN$-neighbor of $C_j$ for $j \neq i$ as in~\cref{fig:quadruple}. 

\tikzset{
Big dot/.style={
circle, inner sep=0pt, 
minimum size=2.5mm, fill=black
}
}
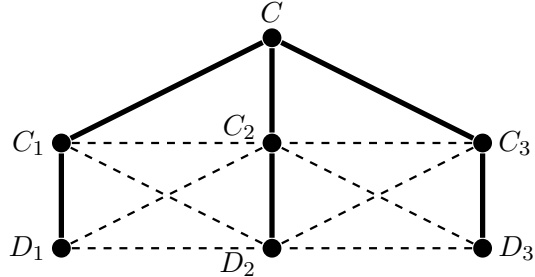
\begin{figure}[htb]
\begin{center}
\begin{tikzpicture}[thick, scale=0.7]
%kite
\node[Big dot] (C) at (0,4) {};
\draw (0,4.1) node[above] {$C$};
\node[Big dot] (C1) at (-4,2) {};
\draw (-4.1,2) node[left] {$C_1$};
\node[Big dot] (C2) at (0,2) {};
\draw (-0.1,2.3) node[left] {$C_2$};
\node[Big dot] (C3) at (4,2) {};
\draw (4.1,2) node[right] {$C_3$};
\node[Big dot] (D1) at (-4,0) {};
\draw (-4.1,0) node[left] {$D_1$};
\node[Big dot] (D2) at (0,0) {};
\draw (-0.1,-0.3) node[left] {$D_2$};
\node[Big dot] (D3) at (4,0) {};
\draw (4.1,0) node[right] {$D_3$};

\draw [line width=0.7mm] (C) -- (C1);
\draw [line width=0.7mm] (C) -- (C2);
\draw [line width=0.7mm] (C) -- (C3);
\draw [line width=0.7mm] (D1) -- (C1);
\draw [line width=0.7mm] (D2) -- (C2);
\draw [line width=0.7mm] (D3) -- (C3);
\draw [line width=0.3mm,dashed] (C1) -- (C2);
\draw [line width=0.3mm,dashed] (C1) -- (D2);
\draw [line width=0.3mm,dashed] (D1) -- (D2);
\draw [line width=0.3mm,dashed] (D1) -- (C2);
\draw [line width=0.3mm,dashed] (C3) -- (C2);
\draw [line width=0.3mm,dashed] (C3) -- (D2);
\draw [line width=0.3mm,dashed] (D3) -- (D2);
\draw [line width=0.3mm,dashed] (D3) -- (C2);
\end{tikzpicture}	
\end{center}
\caption{Subgraph induced by heavy clusters $C$, $C_i$, and $D_i$ in $H_{\cN}$. Bold edges represent heavy touching sets, while dashed edges represent the remaining neighboring relations in question that cannot be heavy.
}\label{fig:quadruple}
\end{figure}

Observe that if 
on at least one side of \cref{fig:quadruple} all touching sets represented by dashed edges are empty, then we recover the first case.
Hence, the only remaining case to consider is that 
at least one of the touching sets $\tau(C_1,C_2)$, $\tau(D_1,C_2)$, $\tau(C_1,D_2)$, $\tau(D_1,D_2)$ is nonempty, and at least one of the touching sets $\tau(C_2,C_3)$, $\tau(D_2,C_3)$, $\tau(C_2,D_3)$, $\tau(D_2,D_3)$ is nonempty. Assume this is the case. Then inserting three edges (the first between $C_1$ and $D_1$, the second between $C_2$ and $D_2$, and the last one in the left side's  nonempty touching set) yields a single cluster $C'$ that contains ${C_1, D_1,C_2,D_2}$ and does not perturb any other cluster $\go$ or touching classes in $\htau_\go$. 
Note that after such an operation the touching classes $\htau_\go(C_1,C)$ and $\htau_\go(C_2,C)$ are combined, giving a single class along which $C'$ has $\ge 2$ $\w$-nonvanishing ends.
On the right side, inserting two edges (the first between $C_3$ and $D_3$, and the second in the respective nonempty touching set) gives a cluster $C''$ that has $\ge 3$ $\w$-nonvanishing ends along the touching class $\htau_\go(C'',C)$.
\end{proof}

%%%%%%%%%%%%%%%%%%%%%%%%%%%%%%%%%%%%%
\subsection{Proof of \cref{non_hf_class}}\label{sec:non_hf_class}
%%%%%%%%%%%%%%%%%%%%%%%%%%%%%%%%%%%%%

\begin{comment}
\begin{thm}\label{non_hf_class}
There is a jointly invariant random process $(\go,\htau_\go, H(\htau_\go))$, where $\go$ is a configuration of a Bernoulli$(p)$ bond percolation, $\htau_\go$ is a partition of $\go$ into the touching classes, and $H(\htau_\go)$ is a connected (locally countable) graph on each $\htau_\go$-class, with the property that restricting to the event where there is a $\htau_\go$-class with $\ge 3$ $\w$-nonvanishing ends in the corresponding percolation cluster of $\go$ the random graph $H(\htau_\go)$ is not hyperfinite.  
\end{thm}
\end{comment}

As we mentioned previously, in the corresponding step of the proof of the cluster repulsion for unimodular transitive graphs, \cite[Lemma 2.5]{Timar06neighb}, Tim\'ar constructs an invariant forest on a set of the trifurcations with high expected degree and then applies \cite[Proposition 7.2]{BLPS99inv}. 
The correct analogue of this proposition in the nonunimodular/mcp setting is the generalized Adams' Dichotomy from \cite{AnushRobin}, which characterizes hyperfiniteness in acyclic mcp graphs as those with at most $ 2$ nonvanishing ends in almost every component.

Consequently, the analogue of Tim\'ar's proof would require building a forest on trifurcation points with more than $2$ nonvanishing ends. This poses significant additional difficulty: while the \emph{local} property of having expected degree larger than $2$ guarantees the existence of more than $2$ ends, we are not aware of any local property that would guarantee more than $2$ nonvanishing ends.
We thus take a completely different approach and show that $\htau_\go$ is not hyperfinite directly as an IRP (without placing an additional graph structure on it). To do this, we use the cluster graphing construction to transfer to the setting of countable Borel equivalence relations and apply \cref{thm:nonamenable}.

\begin{proof}[Proof of \cref{non_hf_class}]
Consider a pair $(\cG^{\mathrm{cl}},\cS)$, where $\cG^{\mathrm{cl}}$ is a cluster graphing of $\go$ that factors onto $\htau_\go$ and $\cS$ is the subequivalence relation induced by $\htau_\go$. Invoking the family of component-wise isomorphisms $\{\varphi_x\}_{x\in X}$ of (relatively) weighted graphs from Proposition~\ref{graphing-and-rg}, \cref{infinitely many nonvanishing ends} yields that $\cG^{\mathrm{cl}}$ has infinitely many $\w_\nu$-nonvanishing ends along $\cS$ on some $\cS$-invariant subset of $Y$ with $\nu$-positive measure.  
By \cref{thm:nonamenable}, on this event $\cS$ is $\nu$-nonhyperfinite. The conclusion now follows from \cref{prop:hf_properties}\ref{prop:hf_properties_cl}.
\end{proof}

%%%%%%%%%%%%%%%%%%%%%%%%%%%%%%%%%%%%%%%%%%%%%%%%%%%%%%%%%%%%%%%%%%%%
%
\section{Subgraphs with positive weighted Cheeger constant}
%
%%%%%%%%%%%%%%%%%%%%%%%%%%%%%%%%%%%%%%%%%%%%%%%%%%%%%%%%%%%%%%%%%%%%

The following result is a translation of Kaimanovich's theorem, as in Theorem~\ref{thm:Kai}, to the setting of percolation theory.

\begin{thm}[Essentially~\cite{kaimanovich1997amenability}]\label{thm:pos_Cheeger}
Let $G$ be a connected locally finite quasi-transitive graph and let $\w$ be the relative weight function induced by a closed subgroup $\Gamma\subseteq\Aut(G)$ as in \eqref{def:haarweights}. Let $\go$ be a $\Gamma$-invariant random graph on $V$ with uniformly bounded degrees a.s.
Then $\go$ is hyperfinite if and only if for any random nonempty subset $A\subseteq V(G)$ such that the law of the pair $(\go,A)$ is jointly $\Gamma$-invariant, every connected component $C$ of the induced subgraph $\go'=\go|_A$ has $\Phi_V^{\w}(C)=0$ a.s.
\end{thm}

\begin{remark}[Bounded degree assumption]
In \cref{thm:pos_Cheeger} we assume that $\mathbf{P}$-a.e.~configuration has uniformly bounded degree. We do so for sheer convenience to match the assumptions of \cite{kaimanovich1997amenability}. The result holds more generally for any locally finite process. However, the proof would require an additional step as the analogous result for locally finite Borel graphs \cite[Theorem 1.1]{CGMTDoneended} is stated in terms of the ``global" Cheeger constant. Hence one needs to show the equivalence of the local and global Cheeger constants, which is done in \cite[Theorem 2.31]{TasminThesis}. 
\end{remark}

\begin{proof}[Proof of \cref{thm:pos_Cheeger}]

To simplify the exposition, we assume that $\Gamma$ acts transitively on $V$ because the quasi-transitive case is treated analogously.

$\Leftarrow:$ Recalling the notation from \cref{sec:cluster_graphing}, let $\cG^{\mathrm{cl}}_\go$ be a cluster graphing of $\mathbf{P}$ on a standard probability space $(Y,\nu)$.
By \cref{prop:hf_properties}\labelcref{prop:hf_properties_cl}, it is enough to show that $\clustG$ is $\nu$-hyperfinite, for which it is enough to fix an arbitrary measurable subset $B \subseteq Y$ of $\nu$-positive measure and show that $\Phi^{\w_\nu}_{\wt{x}} (\clustG |_B) = 0$, by \cref{thm:Kai}.

To this end, we will push-forward $B$ into a random subset $A \subseteq V$ that is jointly $\Gamma$-invariant with $\mathbf{P}$, which will yield $\Phi^{\w_\nu}_{\wt{x}} (\clustG |_B) = 0$, by the hypothesis.
More precisely, recall the component-wise isomorphisms $\varphi_x$ defined in~\eqref{eq:phi_bijection} and 
let $\sigma_B : X \to 2^V$ be defined by $x \mapsto \varphi_x(B \cap [\pi(x)]_{\wt\cG})$.
Note that $\sigma_B$ is $\Gamma$-equivariant because $x \mapsto \varphi_x$ is $\Gamma$-equivariant (as in~\eqref{eq:phi_is_equivariant}) and $\pi$ is a reduction of the orbit equivalence relation $\cR_\Gamma$ of $\Gamma$ on $X$ to $\cR_{\wt \cG}$.
Thus, the pushforward of the measure $\mu$ via factor maps $(\go, \sigma_B)$ is a jointly $\Gamma$-invariant random process $(\go, A)$ on $2^{V \times V} \times 2^V$.
By the hypothesis, $\Phi^\w_V (\go |_A) = 0$ a.s. 
But for each $x \in X$, the map $\varphi_x$ is an isomorphism of (relatively) weighted graphs $\clustG |_{B \cap [\pi(x)]_{\wt \cG}}$ and $\go(x) |_{\sigma_B(x)}$, so $\Phi^{\w_\nu}_{\pi(x)} (\clustG |_B) = 0$ for a.e.\ $x$.

$\Rightarrow:$ For $\mathbf{P}$ and $\mathbf{P}'$ as in the statement, let $\{\mathbf{P}_n\}_{n\in\bN}$ be a hyperfinite exhaustion of $\mathbf{P}$ and let $\mathcal{G}^{\mathrm{cl}}$ be a cluster graphing on $(Y, \nu)$ that factors onto $(\mathbf{P},\{\mathbf{P}_n\}_{n\in\bN}, \mathbf{P}')$, with factor map given by $x \mapsto (\go(x), (\go_n(x))_{n \in \N}, A(x))$. 
Define
\[
B_X := \{x \in X : o \in A(x)\}
\]
and let $B := \pi(B_X)$. 
Since $B_X$ and $\pi$ are Borel, $B$ is analytic and hence measurable.
Observe that for each $x \in X$, we have $\sigma_B(x) = A(x)$, so $\Phi^{\w_\nu}_{\pi(x)} (\clustG |_B) = 0$ for a.e.\ $x$ because the map $\varphi_x$ is an isomorphism of (relatively) weighted graphs $\clustG |_{B \cap [\pi(x)]_{\wt \cG}}$ and $\go(x) |_{A(x)}$, and $\Phi^\w_V (\go(x) |_{A(x)}) = 0$ a.e.\ by the hypothesis.
\end{proof}

A corollary of \cref{thm:pos_Cheeger} is a weighted generalization of \cite[Theorem 1.1]{BLS-perturb} as in \cref{thm:pos_Cheeger_Bernoulli}.
\begin{proof}[Proof of \cref{thm:pos_Cheeger_Bernoulli}]
By \cite[Theorem 6.6]{wamen}, any Bernoulli$(p)$ percolation process on a $\w$-nonamenable graph $G$ that contains a heavy cluster is not hyperfinite. The conclusion now follows from \cref{thm:pos_Cheeger}. 
\end{proof}

\begin{defn}
Let $K$ be a connected locally finite graph that is equipped with a weight function $\w: V(K)\to\bR^+$. Let $\ell: E(K)\to \bN$ be a labeling of edges of $K$ and let $d_\ell$ be the induced metric. We say that $K$ has \textbf{exponential growth of weights} with respect to $d_\ell$ if 
\begin{equation*}
\liminf_{n\to\infty} \frac{1}{n}\log \w\left(S_K(o,n)\right)>0,
\end{equation*}
where $S_K(o,n)$ denotes a sphere of radius $n$ with respect to the metric $d_\ell$ around some vertex $o$ in $K$. If the metric is not mentioned then we mean the natural graph metric of $K$.
Similarly, we say that a subset of vertices $W\subseteq K$ exhibits \textbf{exponential growth of weights} in $K$ if the same holds for $S_K(o,n)\cap W$. 
\end{defn}

\begin{lem}\label{pos_Cheeger_implies_weighted_grwth}
Let $K$ be a connected locally finite graph equipped with a weight function $\w: V(K) \rightarrow \bR^+$. If $\Phi^{\w}_V(K) > 0$, then $K$ exhibits exponential growth of weights (with respect to the graph metric of $K$).
\end{lem}

\begin{proof}
Fix $o \in V(K)$. Let $B_n(o) = \{x \in V(K): d(x,o) \leq n\}$.
It suffices to show that for all $n \geq 1,$
\[\w(S_n(o)) \geq \Phi_V^\w(K) \w(o) (1+\Phi_V^\w(K))^{n-1}. \]
For the base case, notice that $\{o\}$ is a finite set, thus $\w(S_1(o)) \geq \Phi_V^\w(K) \w(o)$ by the definition of $\Phi_V^\w(K)$.
Now suppose the claim is true for $n$. Then by the definition of $\Phi_V^\w(K)$, we have
\begin{align*}
\w(S_{n+1}(o)) &\geq \Phi_V^\w(K) \w(B_n(o))  \\
&= \Phi_V^\w(K) \sum_{i=0}^{n} \w(S_i(o)) \\
&= \Phi_V^\w(K) \bigg(\w(o) + \Phi_V^\w(K) \sum_{i=1}^{n} \w(S_i(o))\bigg) \\
&\geq \w(o) \Phi_V^\w(K) \bigg(1 + \Phi_V^\w(K) \sum_{i=1}^{n} (1 + \Phi_V^\w(K))^{i-1}\bigg) \\
&= \Phi_V^\w(K)\bigg(1 + \Phi_V^\w(K) \cdot \dfrac{1 - (1 + \Phi_V^\w(K))^{n-1}}{1 - (1+ \Phi_V^\w(K))} \bigg)\\
&= \Phi_V^\w(K)\w(o) (1 + \Phi_V^\w(K))^{n-1}.
\end{align*}
This completes the proof.
\end{proof}

Analogously to \cite[Lemma 2.5]{Timar06neighb}, to prove \cref{thm:heavyrepulsion} we will need to establish the exponential growth of weights in a metric induced by an edge labeling. In the unimodular setting, it is stated in \cite{Timar06neighb} that if the labels are bounded, then exponential labeled growth is equivalent to exponential growth of the underlying graph. However, this is technically not quite the case; indeed, in a $3$-regular infinite tree $T_3$, where each edge is labeled by $2$, the spheres of odd radii are empty, so the $\liminf$ is technically $0$. Of course, this minor issue can be easily fixed by establishing the exponential growth along a sequence of disjoint annuli around a point, which we do in the following lemma.

\begin{lem}\label{lem: annuli}
Let $K$ be a connected locally finite graph equipped with a weight function $\w: V(K) \rightarrow \bR^+$. Let $\ell: E(K)\to \{1, \cdots, N\}$ be a labeling of edges of $K$ for some $N \in \N$, and let $d_\ell$ be the induced metric, while $d_K$ denotes the natural graph metric of $K$.
Assume that $K$ exhibits exponential growth of weights with respect to $d_K$. Consider the sequence $(m_n)_{n \in \N}\subseteq\bN$ defined inductively by $m_1 = 1$, and $m_{n + 1} = N m_n +1$. 

Then the annuli $(A_n)_{n \geq 1}$ defined by $A_{n} = \{x \in K: m_n \leq d_\ell(x,o) \leq N m_n\}$ are disjoint and satisfy
\[\w(A_n) \geq c^{m_n}\]
for some $c>1$ and all $n$ large enough.
\end{lem}

\begin{proof}
Let $i: (K, d_K) \rightarrow (K, d_\ell)$ denote the canonical map between these metric spaces which is the identity map on the vertices. Fix $o \in V(G)$ and let $S_n(o) = \{x \in K: d_K(o, x) = n\}$. Observe that for all $n \geq 1$ and any $x \in i(S_n(o))$, one has $n \leq d_\ell(x,o) \leq nN$. Thus, taking $(m_n)_{n \geq 1}$ as above we get that the family $(A_n)_{n \geq 1}$ is disjoint and $i(S_n(o)) \subseteq A_{n}$. Since $\w(S_n(o)) \geq c^n$ for all $n$ large enough, it follows that $\w(A_n) \geq \w(S_{m_n}(o)) \geq c^{m_n}$ for $n$ large enough. 
\end{proof}

We now have established weighted analogs for all necessary ingredients to apply Tim\'ar's mass transport scheme from \cite{Timar06neighb} in the nonunimodular setting.

\begin{proof}[Second proof of~\cref{thm:heavyrepulsion}]
Consider a jointly invariant random process $(\go,\htau_\go)$, where $\go$ is a configuration of a Bernoulli$(p)$ bond percolation and $\htau_\go$ is a partition of $\go$ into the touching classes.
Assume towards a contradiction that the conclusion of \cref{thm:heavyrepulsion} does not hold. Let $H(\htau_{\go})$ denote an invariant (locally countable) random graph that consists of the complete graphs on $\htau_{\go}$-classes. 
Label each edge in $H(\htau_{\go})$ by the distance in $C$ between its endpoints and let $H_n(\htau_{\go})$ be the subgraph that contains only edges with the label at most $n$. 
Note that by local finiteness of $G$ each vertex has finitely many vertices at distance at most $n$, thus $H_n(\htau_{\go})$ is necessarily locally finite.
By \cref{non_hf_class} $\htau_\go$ is not hyperfinite. 
Since the joint law of $H(\htau_{\go})$ and the labeling is invariant under the diagonal action of the group, this defines an increasing exhaustion of $H(\htau_{\go})$. 
By \cref{prop:hf_properties}\ref{prop:hf_properties_incr_union} there is $N\in \bN$ such that $H_N(\htau_{\go})$ is a nonhyperfinite random graph. 

By \cref{thm:pos_Cheeger} and \cref{pos_Cheeger_implies_weighted_grwth}, $H_N(\htau_{\go})$ exhibits exponential growth of weights (with respect to the graph metric of $H_N(\htau_\go)$). Note that the labeled distance between any two points $x$ and $y$ in the same connected component of $H_N(\htau_{\go})$ is at least $\dist_{\go}(x,y)$. Thus, there is a $\htau_\go$-class $\htau_\go(C,C')$ that exhibits exponential growth of weights in $C$ along a sequence of annuli $(A_n)_{n \geq 1}$ as in \cref{lem: annuli}.

Following \cite{Timar06neighb}, define the following mass transport scheme. Let every vertex $x$ select a geodesic path $\pi(x,y)$ in $C(x)$ to each of its neighbors $y$ in $G$ that are in the same cluster. Enumerate vertices in each geodesic path as $x = x_1, \cdots, x_m = y$, and let $x$ send mass $1/k^2$ to $x_k$ for each $k\in\bN$ and path $\pi(x,y)$. 
This yields an unbalanced tilted mass transport analogously to \cite{Timar06neighb}.
\end{proof}

\bibliographystyle{alpha}
\bibliography{nonunimod} 
\end{document}